\documentclass[12pt]{amsart}

\usepackage{amsmath,amssymb,amsthm}
\usepackage{bm,bbm,comment, mathtools}
\usepackage{graphicx}
\usepackage[top=30truemm, bottom=30truemm,left=25truemm,right=25truemm]{geometry}

\AtBeginDocument{
  \abovedisplayskip     =0.5\abovedisplayskip
  \abovedisplayshortskip=0.5\abovedisplayshortskip
  \belowdisplayskip     =0.5\belowdisplayskip
  \belowdisplayshortskip=0.5\belowdisplayshortskip}

  \theoremstyle{definition}
  \newtheorem{dfn}{Definition}[section]
  \newtheorem*{cond}{Condition (A)}
  \newtheorem{rem}[dfn]{Remark}
  \theoremstyle{plain}
  \newtheorem{prop}[dfn]{Proposition}
  \newtheorem{conj}[dfn]{Conjecture}
  \newtheorem{lem}[dfn]{Lemma}
  \newtheorem{thm}[dfn]{Theorem}
  \newtheorem*{thm*}{Theorem}
  \newtheorem*{lem*}{Lemma}
  \newtheorem{cor}[dfn]{Corollary}

\renewcommand{\labelenumi}{(\theenumi)}

\numberwithin{equation}{section}
\begin{document}

\title[Kurokawa-Mizumoto congruence and differential operators]{Kurokawa-Mizumoto congruence and differential operators on automorphic forms}
\author{Nobuki TAKEDA}
\address{Department of Mathematics, Graduate School of Science, Kyoto University, Kyoto 606-8502, Japan}
\email{takeda.nobuki.72z@st.kyoto-u.ac.jp}
\date{}
  \maketitle
  \begin{abstract}
    We give the sufficient conditions for the vector-valued Kurokawa-Mizumoto congruence
    related to the Klingen-Eisenstein series to hold.
    And we give a reinterpretation for differential operators on automorphic forms by the representation theory.
  \end{abstract}
  
  \section{Introduction}
  Let $k$, $\nu$ be even integers,
  and $f(z) =\sum_{n>0}a(n,f)\mathbf{e}(nz)$ be a (degree 1) normalized Hecke eigenform of weight $k+\nu$.
  Here $\mathbf{e}(x) \coloneqq  e^{2\pi i x}$ for $x \in \mathbb{C}$.
  Then, it is conjectured that for a sufficiently large prime ideal $\mathfrak{p}$ of Hecke field $\mathbb{Q}(f)$ for $f$
  associated with the special value of the L-function $L(k-1,f,\mathrm{St})$,
  there is a some degree 2 Siegel cusp form $F$ of weight $\det^k\otimes\mathrm{Sym}^\nu$
  and that the following congruence holds:
  \[\lambda_F(p)=(1+p^{k-2})a(p,f) \mod \mathfrak{p}',\]
  where $\lambda_F(p)$ is the eigenvalue of the Hecke operator $T(p)$ on $F$,
  and $\mathfrak{p}'$ is a prime ideal of $\mathbb{Q}(f,F)=\mathbb{Q}(f)\cdot\mathbb{Q}(F)$ lying above $\mathfrak{p}$.
  Let $[f]^{(k+\nu,k)}$ be a Klingen-Eisenstein lift of $f$ to a degree 2, weight $\det^k\otimes\mathrm{Sym}^\nu$.
  Since $\lambda_{[f]^{(k+\nu,k)}}=(1+p^{k-2})a(p,f)$ due to Arakawa \cite{arakawa1983vector},
  This congruence is between a Klingen-Eisenstein lift and another Siegel modular form.
  Congruences of this type are proved initially by Kurokawa \cite{kurokawa1979congruences} for $k= 20, \nu=0$,
  then by Mizumoto \cite{mizumoto1986congruences} for $k=22, \nu=0$,
   and by Satoh \cite{satoh1986certain} for $\nu=2$, and sufficient conditions examined by Katsurada-Mizumoto \cite{katsurada2012congruences} for $\nu=0$.
  We call the congruences of this type Kurokawa-Mizumoto congruences.
  
  We give sufficient conditions for the vector-valued Kurokawa-Mizumoto congruence.
  This result proves Conjecture 10.6 of Begstr\"{o}m, Faber and van der Geer \cite{bergstrom2014siegel}.
  \begin{thm*}[Theorem~\ref{thm:main}]
      Let $k,\nu$ be positive even integers with $k\geq 6$,
      $f_{1,1}=f, \ldots, f_{1,d_1}$ be a basis of $S_{k+\nu}\left(\Gamma_1\right)$ consist of normalized Hecke eigenforms,
      $p$ be a prime number of $\mathbb{Q}$
      and $A \in H_{2}(\mathbb{Z})_{>0}$ be a half-integral positive definite matrix of degree $2$.
      Suppose that $A$ and $p$ satisfy the following conditions:
      \begin{enumerate}
        \item $\mathrm{ord}_p(\mathbb{L}(k-1,f,\mathrm{St}))=:\alpha>0$,
        \item $\mathrm{ord}_p(\mathcal{C}_{4,k}(f)a(A,[f]^{(k+\nu,k)}))= 0$,
        \item $p\geq2(k+\nu)-3$.
      \end{enumerate}
      Then, there is a Hecke eigenform $G \in M_{\rho_2}(\Gamma_2)$
       such that G is not a scalar multiple of $\left[f\right]^{(k+\nu,k)}$ and
      \[\left[f\right]^{(k+\nu,k)}\equiv_{ev} G \mod \mathfrak{p}\]
      for some prime ideal $\mathfrak{p}\mid p $ of $\mathbb{Q}(G)$.
      If $\mathrm{ord}_p(\gamma_1) = 0$, Condition (3) can be changed to Condition (3)':
      \begin{enumerate}
        \renewcommand{\labelenumi}{(\arabic{enumi})'}
        \setcounter{enumi}{2}
        \item $p\geq \mathrm{max}\left\{2k,k+\nu-2\right\}$.
      \end{enumerate}
    
      If moreover $k\geq 6$ and $p$ satisfy the following conditions:
      \begin{enumerate}
        \renewcommand{\labelenumi}{(\arabic{enumi})}
        \setcounter{enumi}{3}
        \item $\mathrm{ord}_p(\mathbb{L}(k-1,f_{1,i},\mathrm{St}))\leq 0 \ (2\leq i \leq d_1)$,
        \item $p$ is coprime with every $\mathfrak{A}(f_{r,i})$ ($1\leq r \leq2$, $1\leq i \leq d_r$).
      \end{enumerate} 
      there is a Hecke eigenform $G \in S_{\rho_2}(\Gamma_2)$
       such that G is not a scalar multiple of $\left[f\right]^{(k+\nu,k)}$ and 
      \[\left[f\right]^{(k+\nu,k)}\equiv_{ev} G \mod \mathfrak{p}^\alpha\]
      for some prime ideal $\mathfrak{p}\mid p $ of $\mathbb{Q}(G)$.
  \end{thm*}
  The former part  of the proof of this theorem follows Atobe-Chida-Ibukiyama-Katsurada-Yamauchi's method \cite{atobe2023harder},
  and the key to the proof of the main theorem is the following lemma \cite[Lemma 6.10]{atobe2023harder}.
  \begin{lem*}[Lemma~\ref{lem:cong}]
    Let $F_1,\ldots,F_d$ be Hecke eigenforms in $M_{\mathbf{k}}(\Gamma_n)$ linearly independent over $\mathbb{C}$.
    Let $K$ be the composite field $\mathbb{Q}(F_1),\ldots,\mathbb{Q}(F_d)$, $\mathcal{O}$ the ring of integers in $K$
    and $\mathfrak{p}$ a prime ideal of $K$. Let $G(Z) \in(M_{\rho_n}(\Gamma_n)\otimes V_{n,\mathbf{k}})(\mathcal{O}_{(\mathfrak{p})})$ and assume the following conditions
    \begin{enumerate}
      \item $G$ is expressed as
      \[ G(Z)=\sum_{i=1}^d c_i F_i(Z)\]
      with $c_i \in V_{n,\mathbf{k}}$.
      \item $c_1a(A,F_1) \in (V_{n,\mathbf{k}} \otimes V_{n,\mathbf{k}})(K)$ and $\mathrm{ord}_\mathfrak{p}(c_1a(A,F_1))<0$
      for some $A \in H_n(\mathbb{Z})$.
    \end{enumerate}
    Then there exists $i\neq 1$ such that
    \[F_i \equiv_{ev}F_1 \mod \mathfrak{p}.\]
  \end{lem*}
  
  The latter part of the proof of the main theorem is due to Katsurada-Mizumoto's method \cite{katsurada2012congruences}.
  
  To obtain the main theorem,
  we take as $G(Z)$ in this corollary a function constructed from the Siegel-Eisenstein series $E_{n, \mathbf{k}}(Z,s)$
  and using the pullback formula (Theorem \ref{thm:eisen}), 
  we show that a decomposition of $G(Z)$ that satisfies the conditions of the corollary can be obtained.
  We also showed in Theorem~\ref{thm:cusp}, using Galois representations,
  the conditions under which G in Theorem~\ref{thm:main} can be taken as a cusp form.
  
  In the course of the proof,
  we will discuss using differential operators that preserve the automorphic properties,
  which are deeply investigated by Ibukiyama \cite{ibukiyama1999differential,Ibukiyama2020generic}.
  In this paper, we also reinterpreted this differential operator by using Howe duality in representation theory.
  The Proof is given using Ban's method \cite{ban2006rankin}.
  \begin{thm*}[Theorem~\ref{thm:mainweight}]
      Let F be a holomorphic automorphic form of weight $\tfrac{k}{2}\mathbbm{1}_{n}$ for $\Gamma_n$
      and $D \in \Delta_{n,k}$ be a Young diagram  such that $l(D)\leq\min\{n_1,\ldots,n_d\}$. 
      We put $\partial_Z=\left(\frac{1+\delta_{i,j}}{2}\frac{\partial}{\partial z_{i,j}}\right)$.
      We denote by $\mathrm{Res}$ the restriction of a function
      on $\mathbb{H}_n$ to $\mathbb{H}_{n_1} \times\cdots\times \mathbb{H}_{n_d}$.
      Then we have
        $\mathrm{Res}\left(\Phi_{h}(\partial_Z)F\right)\in \bigotimes_{s=1}^{d}  M_{\tau_{n_s,k}(D)}(\Gamma_{n_s})$ 
         for any $h \in$ {\small $\left(\left(\bigotimes_{s=1}^{d} \mathfrak{H}_{n_s,k}(D)\right)^{\mathrm{O}_k}
        \otimes \left(\bigotimes_{s=1}^{d} U^*_{\tau_{n_s,k}(D)}\right)\right)^{\widetilde{K'}}$}.
  \end{thm*}
    
  \setcounter{section}{1}
  \vskip\baselineskip
  
  This paper is organized as follows:
  In Section 2, We explain the Siegel modular form, the Hecke operator,
  and the terminology used in this paper.
  In Section 3, we define the Siegel operator and Klingen-Eisenstein series
  that raise and lower the degree of Siegel modular forms and review their properties. 
  In Section 4, we state the results of differential operators
  that preserve the automorphic properties and give an interpretation by representation theory of the differential operator.
  Then,  we explain the Pullback formula, which is the key to the proof of the main theorem.
  In Section 5, we prove the main theorem.
  In Section 6, we consider the conditions that appear in the main theorem (Theorem \ref{thm:main}).
  In section 7, we give explicit examples of Kurokawa-Mizumoto congruences.
  
  \textbf{Notation.}
  For a commutative ring $R$, we denote by $R^\times$ the unit group of $R$.
  We denote by $M_{m,n}(R)$ the set of $m\times n$ matrices with entries in $R$.
  In particular, we denote $M_n(R)\coloneqq M_{n,n}(R)$.
  Let $\mathrm{GL}_n(R) \subset M_n(R)$ be a general linear group of degree $n$.
  For an element $X \in M_n(R)$, we denote by $X>0$ (resp. $X \geq 0$) X is a positive definite matrix
  (resp. a non-negative definite matrix).
  For a subset $S \subset M_n(R)$, we denote by $S_{>0}$ (resp. $S_{\geq 0}$) the  subset of positive definite
  (resp. non-negative definite) matrices in $S$.
  Let $\det^k$ be the 1 dimensional representation of multiplying $k$-square of determinant for $\mathrm{GL}_n(\mathbb{C})$,
  and  $\mathrm{Sym}^\nu$ be the degree $\nu$ symmetric tensor product representation of $\mathrm{GL}_2(\mathbb{C})$.
  
  Let $K$ be an algebraic field, and $\mathfrak{p}$ be a prime ideal of K.
  We denote by $K_\mathfrak{p}$ a $\mathfrak{p}$-adic completion of $K$
  and by $\mathcal{O}_K$ the integer ring of $K$.
  We denote by $\mathrm{ord}_{\mathfrak{p}}(\cdot)$
  the additive valuation of $K_\mathfrak{p}$ normalized so that
  $\mathrm{ord}_{\mathfrak{p}}(\varpi)=1$ for a prime element $\varpi$ of $K_{\mathfrak{p}}$.
  Let $p_\mathfrak{p}$ be the prime number such that $p_\mathfrak{p}\mathbb{Z}=\mathbb{Z}\cap\mathfrak{p}$.
  
  If a group $G$ acts on a set $V$ then, we denote by $V^G$ the $G$-invariant subspace of $V$.

  For a representation $(\tau, U)$, we denote by $(\tau^*, U^*)$ the contragredient representation of $(\tau, U)$.
  
  \section{Siegel Modular Forms}
  Let $\mathbb{H}_n$ be the Siegel upper half space of degree $n$, that is
  \[ \mathbb{H}_n = \{Z \in M_n(\mathbb{C}) | Z= {}^t\!Z=X+\sqrt{-1}Y, \ X,Y\in M_n(\mathbb{R}),\ Y >0\}.\]
  We put $\Gamma_n=\mathrm{Sp}_n(\mathbb{Z})=\{M \in \mathrm{GL}_{2n}(\mathbb{Z}) | MJ_n{}^t\!M=J_n\}$,
  where $J_n=\begin{pmatrix}O_n & I_n \\ -I_n & O_n \\ \end{pmatrix}$.\\
  Let $(\rho, V )$ be a polynomial representation of $\mathrm{GL}_n(\mathbb{C})$ on a finite-dimensional complex vector space $V$,
  and take a Hermitian inner product on V such that
  \[(\rho(g)v,w)=(v,\rho({}^t\!\bar{g})w).\]
  
  For a $V$-valued function $F$ on $\mathbb{H}_n$,
   and $g= \begin{pmatrix}A& B \\C & D \\\end{pmatrix} \in \Gamma_n$, put $j(g,Z)=CZ+D$, and
  \[ F|_{\rho}[g](Z)=\rho(j(g,Z))^{-1}F((AZ+B)(CZ+D)^{-1}) \quad (Z \in \mathbb{H}_n).\]
  We say that $F$ is a (level 1 holomorphic) Siegel modular form of weight $\rho$
  if $F$ is a holomorphic $V$-valued function on $\mathbb{H}_n$ and $F|_{\rho}[g]=F$ for all $g \in \Gamma_n$.
  (When $n=1$, another holomorphic condition is also needed.)
  
  Let $H_n(\mathbb{Z})_{\geq 0}$ (resp. $H_n(\mathbb{Z})_{>0}$) be
  the set of half-integral non-negative definite (resp. positive definite) matrices of degree $n$.
  A modular form $F \in M_\rho\left(\Gamma_n\right)$ has the Fourier expansion
   \[ F(Z)=\sum_{T\in H_n(\mathbb{Z})_{\geq0}}a(F,T)\mathbf{e}(\mathrm{tr}(TZ)),\]
   where $a(F,T)\in V$, $\mathbf{e}(z)=\exp(2\pi\sqrt{-1}z)$, and tr is the trace of $M_n(\mathbb{C})$.
   If $a(F, T)=0$ unless $T$ is positive definite,
   we say that $F$ is a (level 1 holomorphic) Siegel cusp form of weight $(\rho, V)$.
  We denote by $M_\rho\left(\Gamma_n\right)$ (resp. $S_\rho\left(\Gamma_n\right)$)
   a complex vector space of all modular (resp. cusp) forms of weight $\rho$.
  For $F,G \in M_\rho\left(\Gamma_n\right)$, we can define the Petersson inner product as
  \[ (F,G)=\int_{D } \left(\rho(\sqrt{Y})F(Z),\rho(\sqrt{Y})G(Z)\right)\det(Y)^{-n-1}dZ,\]
  where $Y=\Im(Z)$, $\sqrt{Y}$ is a positive definite symmetric matrix such that $\sqrt{Y}^2=Y$,
  and $D $ is a Siegel domain on $\mathbb{H}_n$ for $\Gamma_n.$
  This integral converges if either $F$ or $G$ is a cusp form. 
  
  \vskip\baselineskip
  
  We fix a basis $\left\{ v_1,\ldots v_r \right\}$ of the representation space $V$,
  which $(\rho|_{\mathrm{GL}_n(\mathbb{Z})},\oplus_{i=1}^r\mathbb{Z}v_i)$
  is a representation of $\mathrm{GL}_n(\mathbb{Z})$,
  and we write $V(\mathbb{Z})=\oplus_{i=1}^r\mathbb{Z}v_i$.
  For a subring $R$ of $\mathbb{C}$, we write $V(R)=V(\mathbb{Z})\otimes_{\mathbb{Z}}R$.
  We denote $M_\rho\left(\Gamma_n\right)(R)$ by the set of all Modular form $F \in M_\rho\left(\Gamma_n\right)$
  which Fourier coefficients $a(F,T)$ of $F$ are in $V(R)$,
  and $S_\rho\left(\Gamma_n\right)(R)$ is defined in the same way.
  
  Let $K$ be a number field and $\mathcal{O}$ be the ring of integers in $K$.
  For a prime ideal $\mathfrak{p}$ of $\mathcal{O}$ and $a =\sum_{i=1}^ra_iv_i\in V(R)$,
  we define the order $\mathrm{ord}_\mathfrak{p}(a)$ of $a$ respect to $\mathfrak{p}$ as
  \[\mathrm{ord}_\mathfrak{p}(a)=\min_{i=1,\ldots, r}\mathrm{ord}_\mathfrak{p}(a_i). \]
  We say that $\mathfrak{p}$ divides $a$ if $\mathrm{ord}_\mathfrak{p}(a)>0$ and write $\mathfrak{p}|a$.
  This order does not change when changing the basis by $GL_r(\mathbb{Z})$.
  
  \vskip\baselineskip
  
  We call a sequence of non-negative integers $\mathbf{k}=(k_1,k_2,\ldots)$ a dominant integral weight
  if $k_i \geq k_{i+1}$ for all $i$, and $k_i=0$ for almost all $i$,
  and the biggest integer $m$ such that $k_m \neq 0$ a depth $l(\mathbf{k})$ of $\mathbf{k}$.
  The set of dominant integral weights with depth less than or equal to $n$ corresponds bijectively to
  the set of irreducible polynomial representations of $\mathrm{GL}_n(\mathbb{C})$.
  We denote the representation of $\mathrm{GL}_n(\mathbb{C})$ corresponds to
  a dominant integral weight $\mathbf{k}$ by $(\rho_{n, \mathbf{k}},V_{n, \mathbf{k}})$,
  and write $M_\mathbf{k}\left(\Gamma_n\right) =M_{\rho_{n, \mathbf{k}}}\left(\Gamma_n\right)$
  and $S_\mathbf{k}\left(\Gamma_n\right) = S_{\rho_{n, \mathbf{k}}}\left(\Gamma_n\right)$.
  When  $\mathbf{k}=(k,k,\ldots,k)$ (i.e. $\rho_\mathbf{k}=\det^k$),
  we also write $M_k\left(\Gamma_n\right)=M_\mathbf{k}\left(\Gamma_n\right)$
  and $S_k\left(\Gamma_n\right)=S_\mathbf{k}\left(\Gamma_n\right) $.
  For a dominant integral weight $\mathbf{k}=(k_1,\ldots,k_n)$ associated to a representation of $\mathrm{GL}_n(\mathbb{C})$
  we define strong depth $s\text{-}l(\mathbf{k})$ of $\mathbf{k}$ as the biggest integer $m$ such that $k_m>k_n$.
  if $\mathbf{k}=(k,\ldots,k)$, set $s\text{-}l(\mathbf{k})=0$.
  
  \vskip\baselineskip
  
  Let $\mathcal{H}_n$ be the Hecke algebra over $\mathbb{Z}$
  associated to the pair ($\Gamma_n$, $\mathrm{GSP}_n(\mathbb{Q})\cap M_{2n}(\mathbb{Z}))$.
  For a subring $R$ of $\mathbb{C}$ we write $\mathcal{H}_n(R)= \mathcal{H}_n\otimes_{\mathbb{Z}} R$.
  We define the action of Hecke algebra $\mathcal{H}_n(\mathbb{C})$ on $M_\mathbf{k}\left(\Gamma_n\right)$ as
  \[T(F) = \nu(g)^{k_1+\cdots+k_n-\frac{n(n+1)}{2}}\sum_{i=1}^{r}F|_{\rho_{n, \mathbf{k}}}g_i\]
  for a modular $F \in M_\mathbf{k}\left(\Gamma_n\right)$
  and an element $T=\Gamma_ng\Gamma_n= \bigsqcup_{i=1}^r \Gamma_ng_i$ (cosets decomposition) $\in \mathcal{H}_n(\mathbb{C})$,
  where $\nu(g) \in \mathbb{Q}$ is a similitude factor of $g \in \mathrm{GSP}_n(\mathbb{Q})$.
  Note that this definition does not depend on the representatives and it is well-defined.
  
  For a positive rational number $a$ and a positive integer $m$,
  we define a Hecke operator $[a]_n, T(m) \in \mathcal{H}_n$ as
  \[[a]_n=\Gamma_n(a \cdot 1_n)\Gamma_n,\]
  \[T(m)=\sum_{\substack{d_1,\ldots,d_n, e_1,\ldots,e_n \\ d_i|d_{i+1},\  e_{i+1}|e_i\\d_ie_i=m}}
  \Gamma_n\mathrm{diag}(d_1,\ldots,d_n,e_1,\ldots,e_n)\Gamma_n.\]
  Then, $\mathcal{H}_n(\mathbb{Q})$ generated over $\mathbb{Q}$ by $T(p)$, $T(p^2)$, and $[p^{-1}]_n$ for all prime $p$,
  and $\mathcal{H}_n$ generated over $\mathbb{Z}$ by $T(p)$ and $T(p^2)$ for all prime $p$.
  
  We say that a modular form $F \in M_\mathbf{k}\left(\Gamma_n\right)$ is Hecke eigenform
  if $F$ is a common eigenfunction of all Hecke operators $T \in \mathcal{H}_n(\mathbb{C})$.
  For a Hecke eigenform $F \in M_\mathbf{k}\left(\Gamma_n\right)$
  we call the field $\mathbb{Q}(F)$ generated by all Hecke eigenvalues over $\mathbb{Q}$ the Hecke field of $F$.
  It is well known that $\mathbb{Q}(F)$ is a finite totally real algebraic field.
  For ease, we write $\mathbb{Q}(F_1,\ldots,F_m)=\mathbb{Q}(F_1)\cdots\mathbb{Q}(F_m)$
  for $F_1,\ldots,F_m \in M_\mathbf{k}\left(\Gamma_n\right)$.
  
  \begin{dfn}
    Let $F,G \in M_\mathbf{k}\left(\Gamma_n\right)$ be Hecke eigenforms, and $\mathfrak{p}$ be prime ideal in $\mathbb{Q}(F,G)$.
    if $\lambda_F(T)\equiv\lambda_G(T)\mod\mathfrak{p}$ for all $T \in \mathcal{H}_n$, we say that $F$ and $G$ are Hecke congruent,
    and denote $F \equiv_{ev} G\mod\mathfrak{p}$.
  
  \end{dfn}
  
  \vskip\baselineskip
  
  Let $F \in M_\mathbf{k}\left(\Gamma_n\right)$ be a Hecke eigenform,
  and for a prime number $p$ we take the Satake $p$-parameters $\alpha_0(p), \alpha_1(p), \ldots, \alpha_n(p)$ of $F$ so that
  \[\alpha_0(p)^2\alpha_1(p)\cdots\alpha_n(p)=p^{k_1+\cdots+k_n-n(n+1)/2}.\]
  We define standard L-function $L(s,F,\mathrm{St})$ by
  \[L(s,F,\mathrm{St})=\prod_p \left((1-\alpha_0(p)p^{-s})\prod_{r=1}^n\prod_{1\leq i_1<\cdots<i_r\leq n}(1-\alpha_0(p)\alpha_{i_1}(p)\cdots\alpha_{i_r}(p)p^{-s})\right).\]
  For a Hecke eigenform $F \in S_k(\Gamma_n)$,
  we define $\mathbb{L}(s,F,\mathrm{St})$ by
  \[\mathbb{L}(s,F,\mathrm{St})=\Gamma_\mathbb{C}(s)\prod_{i=1}^n\Gamma_\mathbb{C}(s+k-i)\frac{L(s,F,\mathrm{St})}{(F,F)}.\]
  
  \begin{prop}[{\cite[Appendix A]{mizumoto1991poles}}]\label{prop:scalar-algebricity}
    Let $F$ be a Hecke eigenform in $S_k\left(\Gamma_n\right)(\mathbb{Q}(F))$.
    We set $n_0=3$ if $n\geq 5$ with $n\equiv 4$ and $n_0=1$ otherwise.
    Let $m$ be a positive integer $n_0\leq m \leq k-n$ such that $m \equiv n \mod 2$.
    Then $\mathbb{L}(m,F,\mathrm{St}) \in \mathbb{Q}(F)$.
  \end{prop}
  
  \section{Siegel Operator and Klingen-Eisenstein series}
  In this section, let $\mathbf{k}=(k_1,k_2,\ldots)$ be a dominant integral weight with depth $m$ and strong depth $m'$,
  and $n \geq r>0$ be positive integers such that $n\geq m$ and $r\geq m'$.
  In addition, put $\mathbf{k}'=(k_1,\ldots,k_r) $.
  
  For a modular form $F \in M_\mathbf{k}\left(\Gamma_n\right)$ and 
  we define the Siegel operator $\Phi=\Phi^n_r$ as
  \[\Phi(F)(Z_1)=\lim_{y\rightarrow\infty}F\left(\begin{pmatrix}Z_1 & O \\ O & \sqrt{-1}y\cdot 1_{n-r} \\ \end{pmatrix}\right)
  =\sum_{T_1\in H_r(\mathbb{Z})_{\geq0}}a\left(F,\begin{pmatrix}T_1 & O \\ O & O \\ \end{pmatrix}\right)\mathbf{e}(\mathrm{tr}(T_1Z_1)).\]
  Then, $\Phi(F)$ belongs to $M_{\mathbf{k}'}\left(\Gamma_r\right)$.
  
  Let $\Delta_{n,r}$ be the subgroup of $\Gamma_n$ defined by 
  \[\Delta_{n,r}=\left\{\begin{pmatrix}* & *\\ O_{n-r,n+r} & *\\ \end{pmatrix}\in \Gamma_n\right\}.\]
  We define the Klingen-Eisenstein series $[F]^\mathbf{k}(Z,s)=[F]^\mathbf{k}_{\mathbf{k}'}(Z,s)$
  of $F \in M_{\mathbf{k}'}\left(\Gamma_r\right)$ as
  \[[F]^\mathbf{k}(Z,s)=\sum_{g\in \Delta_{n,r}\backslash\Gamma_n}
  \left(\frac{\det\Im(Z)}{\det\Im(\mathrm{pr}_r^n(Z))}\right)^s
  F\left(\mathrm{pr}^n_r(Z)\right)|_{\rho_{n, \mathbf{k}}}g \quad (Z \in \mathbb{H}_n, s \in \mathbb{C}),\]
  where $\mathrm{pr}^n_r(Z)=Z_1$ for $Z=\begin{pmatrix}Z_1 & Z_{12} \\ ^tZ_{12} & Z_2 \\ \end{pmatrix}$
  with $Z_1\in \mathbb{H}_r, Z_2\in\mathbb{H}_{n-r}, Z_{12}\in M_{r,n-r}(\mathbb{C})$
  and define the Siegel-Eisenstein series $E_{n, \mathbf{k}}(Z,s)$
  of weight $\mathbf{k}$ with respect to $\Gamma_n$ as
  \[ E_{n, \mathbf{k}}(Z,s)=\sum_{g\in \Delta_{n,0}\backslash\Gamma_n}
  \left(\det\Im(Z)\right)^s|_{\rho_{n, \mathbf{k}}}g \quad (Z \in \mathbb{H}_n, s \in \mathbb{C}).\]
  
  The holomorphy of the Klingen-Eisenstein series and the Siegel-Eisenstein series holds as follows,
  which is due to Shimura \cite{shimura2000arithmeticity} and \cite[Proposition 2.1.]{atobe2023harder}.
  
  \begin{prop}[{\cite[Proposition 2.1.]{atobe2023harder}}]
    Let $k$ be a positive even integer.
    \begin{enumerate}
      \item Suppose that $k \geq \dfrac{n + 1}{2}$ and that neither $k = \dfrac{n +2}{2} \equiv 2 \mod 4$
        nor $k = \dfrac{n + 3}{2}  \equiv 2 \mod 4$. Then $E_{n,k}(Z)=E_{n,k}(Z,0)$ belongs to $M_k(\Gamma_n)$.
      \item Let $\mathbf{k} = (\underbrace{k+\nu,\ldots,k+\nu}_m,\underbrace{k,\ldots,k}_{n-m})$ such that $\nu \geq 0$ and $k > \dfrac{3m}{2} + 1$ and 
        let $f$ be a Hecke eigenform in $S_{k+\nu}(\Gamma_m)$.
        Then $[f]^\mathbf{k}(Z, s)$ can be continued meromorphically
        to the whole $s$-plane as a function of $s$, and holomorphic at $s = 0$.
        Moreover suppose that $k > \dfrac{n+m+3}{2}$.
        Then $[f]^\mathbf{k}(Z)=[f]^\mathbf{k}(Z, 0)$ belongs to $M_\mathbf{k}(\Gamma_n)$.
    \end{enumerate}
  
  \end{prop}
  
  Consider the case $n=2$, $r=1$.
  In this case, the correspondence between the Siegel operator and the Klingen-Eisenstein series
  has been investigated by Andrianov \cite{and} and Arakawa \cite{arakawa1983vector}.
  
  Let $N_{(k+\nu,k)}(\Gamma_2)$ be the orthogonal complement  of $S_{(k+\nu,k)}(\Gamma_2)$
  in $M_{(k+\nu,k)}(\Gamma_2)$ for the Petersson inner product.
  Then, the following isomorphism holds.
  
  \begin{prop}[{\cite[Proposition 1.3.]{arakawa1983vector}}]\label{prop:arakawa1}
    Let $k$, $\nu$ be even integers with $k>4$, $\nu>0$.
  
    Then the space $N_{(k+\nu,k)}(\Gamma_2)$ is isomorphic to $S_{k+\nu}(\Gamma_1)$,
    and the isomorphism is given via the Siegel operator $\Phi^2_1$ and the Klingen-Eisenstein series $[\ \cdot\ ]^{(k+\nu,k)}$. 
  
  \end{prop}
  
  The following facts about the Hecke eigenvalues are known.
  
  \begin{prop}[{\cite[Proposition 3.2.]{arakawa1983vector}}]\label{prop:arakawa2}
    Let $F \in N_{(k+\nu,k)}(\Gamma_2)$ and put $\Phi (F)(z)=f(z) \in S_{k+\nu}(\Gamma_1)$.
    Then, $F$ is a Hecke eigenform for $\mathcal{H}_2$,
    if and only if $f$ is a common eigenform for $\mathcal{H}_1$.
  
    In this situation, set $T(m)F=\lambda_F(m)F$ and $T(m)f=\lambda_f(m)f$ ($m=1,2,\ldots$).
    Then, for any prime $p$, we have
    \[\left\{
      \begin{array}{ll}
        \lambda_F(p)&=(1+p^{k-2})\lambda_f(p),\\
        \lambda_F(p^2)&=(1+p^{k-2}+p^{2k-4})\lambda_f(p^2)+(p-1)p^{2k+\nu-4}.
      \end{array}
    \right. 
      \]
    
    \end{prop}
  
  Now we will state the conjecture about Kurokawa-Mizumoto congruence.
  
  \begin{conj}
    Let $k$, $\nu$ be even integers with $k>4$, $\nu>0$.
  
    Let $f(z) =\sum_{n>0}a(n,f)\mathbf{e}(nz) \in S_{k+\nu}\left(\Gamma_1\right)$
    be a normalized Hecke eigenform (i.e. a Hecke eigenform with $a(1,f)=1$),
    and suppose that a large prime $\mathfrak{p}$ of $\mathbb{Q}$ divides $L(k-1,f, \mathrm{St})$.
    Then, there exist a Hecke eigenform $F \in S_{(k+\nu,k)}(\Gamma_2)$, and a prime ideal $\mathfrak{p}'|\mathfrak{p}$ in $\mathbb{Q}(F)$
    such that 
    \[F \equiv_{ev} \left[f\right]^{(k+\nu,k)} \mod \mathfrak{p}'.\]
    In particular, for all primes $p$
    \[\lambda_F(p) \equiv_{ev} (1+p^{k-2})\lambda_f(p) \mod \mathfrak{p}'.\]
  
  \end{conj}
  
  The key to the proof of the main theorem is the following lemma \cite[Lemma 6.10]{atobe2023harder},
  which is proved in the same way as \cite[Lemma 5.1]{katsurada2008congruence}.
  \begin{lem}\label{lem:cong}
    Let $F_1,\ldots,F_d$ be Hecke eigenforms in $M_{\mathbf{k}}(\Gamma_n)$ linearly independent over $\mathbb{C}$.
    Let $K$ be the composite field $\mathbb{Q}(F_1),\ldots,\mathbb{Q}(F_d)$, $\mathcal{O}$ the ring of integers of $K$
    and $\mathfrak{p}$ a prime ideal of $K$. Let $G(Z) \in(M_{\rho_n}(\Gamma_n)\otimes V_{n,\mathbf{k}})(\mathcal{O}_{(\mathfrak{p})})$ and assume the following conditions
    \begin{enumerate}
      \item $G$ is expressed as
      \[ G(Z)=\sum_{i=1}^d c_i F_i(Z)\]
      with $c_i \in V_{n,\mathbf{k}}$.
      \item $c_1a(A,F_1) \in (V_{n,\mathbf{k}} \otimes V_{n,\mathbf{k}})(K)$ and $\mathrm{ord}_\mathfrak{p}(c_1a(A,F_1))<0$
      for some $A \in H_n(\mathbb{Z})$.
    \end{enumerate}
    Then there exists $i\neq 1$ such that
    \[F_i \equiv_{ev}F_1 \mod \mathfrak{p}.\]
  
  \end{lem}
  
  \section{Differential Operators}
  
  Let $n = n_1+\cdots+n_d \geq 2$ be a positive integer with $n_d\geq\cdots\geq n_1\geq 1$.
  We embed $\mathbb{H}_{n_1}  \times\cdots\times \mathbb{H}_{n_d}$ in $\mathbb{H}_n$
  and embed $\Gamma_{n_1}\times\cdots\times\Gamma_{n_d}$ in $\Gamma_n$ diagonally.
  Let $\mathbf{k}$ be a dominant integral weight with $l(\mathbf{k})\leq n_1$
  and $\left(\rho_{n_i,\mathbf{k}},V_i\right)$ be the irreducible representation of $\mathrm{GL}_{n_i}(\mathbb{C})$ associated to $\mathbf{k}$.
  For irreducible representations
  $(\rho_i, V_i)$ of $\mathrm{GL}_{n_i}(\mathbb{C})$,
  a $V_1\otimes \cdots\otimes V_d$-valued function $f(Z_1,\ldots,Z_d)$ on $\mathbb{H}_{n_1}  \times\cdots\times \mathbb{H}_{n_d}$
  and $g_i=\begin{pmatrix}A_i & B_i\\ C_i & D_i\\ \end{pmatrix}\in \mathrm{Sp}_{n_i}(\mathbb{R})$, we put
  \[f|_{\rho_1\otimes\cdots\otimes\rho_d}[g_1,\ldots, g_d]=\rho_1(C_1Z_1+D_1)^{-1}\otimes\cdots\otimes\rho_d(C_dZ_d+D_d)^{-1}f(g_1Z_1,\ldots,g_dZ_d).\]
  We consider $V_{\mathbf{k},n_1,\ldots,n_d}\coloneqq V_{\mathbf{k},n_1}\otimes\cdots\otimes V_{\mathbf{k},n_d}$-valued differential operators $\mathbb{D}$ on scalar-valued functions of $\mathbb{H}_n$,
  satisfying Condition (A) below on automorphic properties:
  
  \begin{cond}We fix $k$ and $\mathbf{k}$.
    For any holomorphic function $F$ on $\mathbb{H}_n$
    and any $(g_1,\ldots, g_d) \in \mathrm{Sp}_{n_1}(\mathbb{R}) \times\cdots\times \mathrm{Sp}_{n_d}(\mathbb{R})\subset\mathrm{Sp}_{n}(\mathbb{R})$,
    the operator $\mathbb{D}$ satisfies
  \[\mathrm{Res}(\mathbb{D}(F|_k[(g_1,\ldots, g_d)]))
  =(\mathrm{Res}\ \mathbb{D}(F))|_{\det^k\rho_{n_1,\mathbf{k}}\otimes\cdots\otimes\det^k\rho_{n_d,\mathbf{k}}}[g_1,\ldots,g_d]\]
  where $\mathrm{Res}$ means the restriction of a function on $\mathbb{H}_n$ to $\mathbb{H}_{n_1} \times\cdots\times \mathbb{H}_{n_d}$ .
  \end{cond}
  
  This Condition (A) corresponds to Case (I) in \cite{ibukiyama1999differential}.
  The other Case (II) in \cite{ibukiyama1999differential} is a generalization of the Rankin-Cohen operators in \cite{Cohen1975sums},
  and representation-theoretic reinterpretation was given by Ban \cite{ban2006rankin}.
  Using the method of Ban,
  we give representation-theoretic interpretation of the differential operators satisfying Condition (A) in this section. 
  
  \subsection{Howe duality}
  
  Let $G_n=\mathrm{Sp}_n(\mathbb{R})=\{M \in \mathrm{GL}_{2n}(\mathbb{Z}) | MJ_n{}^t\!M=J_n\}$ be the symplectic group
  and $\widetilde{G_n}=\mathrm{Mp}_n(\mathbb{R})$ be the metaplectic group, which is the double cover of $G_n$.
  $G_n$ acts on $\mathbb{H}_n$ in the same way that $\Gamma_n$ does.
  Let $K_n$ be the stabilizer of $\sqrt{-1} \in \mathbb{H}_n$ in $G_n$.
  Then, $K_n$ is a maximal compact subgroup of $G_n$ and isomorphic to the unitary group $\mathrm{U}_n(\mathbb{C})$
  , which is given by $\begin{pmatrix}A & B\\ -B & A\\ \end{pmatrix}\mapsto A-\sqrt{-1}B$. 
  We take a maximal compact subgroup $\widetilde{K_n}$ of $\widetilde{G_n}$ by the inverse image of $K_n$ to $\widetilde{G_n}$.
  
  We put $\mathfrak{g}_n=\mathrm{Lie}(\widetilde{G_n})$, $\mathfrak{k}_n=\mathrm{Lie}(\widetilde{K_n})$ and
  let $\mathfrak{g}_n=\mathfrak{k}_n\oplus\mathfrak{p}_n$ be the Cartan decomposition. We put
  \[
    \kappa_{i,j}=\mathfrak{c}\begin{pmatrix}e_{i,j} & 0\\ 0 & -e_{j,i}\\ \end{pmatrix}\mathfrak{c}^{-1},\quad
    \pi^{+}_{i,j}=\mathfrak{c}\begin{pmatrix}0 & e_{i,j}+e_{j,i}\\ 0 & 0\\ \end{pmatrix}\mathfrak{c}^{-1},\quad \mathrm{and} \quad
    \pi^{-}_{i,j}=\mathfrak{c}\begin{pmatrix}0 & 0\\ e_{i,j}+e_{j,i} & 0\\ \end{pmatrix}\mathfrak{c}^{-1},
  \]
  where $\mathfrak{c}=\dfrac{1}{\sqrt{2}}\begin{pmatrix}1 & \sqrt{-1}\\ \sqrt{-1} & 1\\ \end{pmatrix}\in \mathrm{M}_{2n}(\mathbb{C})$ and
  $e_{i,j} \in \mathrm{M}_{n,n}(\mathbb{C})$ is the matrix whose only non-zero entry is 1 in $(i,j)$-component.
  $\{\kappa_{i,j}\}$ is a basis of $\mathfrak{k}_{n,\mathbb{C}}$.
  Let $\mathfrak{p}^+_n$ (resp. $\mathfrak{p}^-_n$) be the $\mathbb{C}$-span of
  $\{\pi^{+}_{i,j}\}$ (resp. $\{\pi^{-}_{i,j}\}$) in $\mathfrak{g}_{n,\mathbb{C}}$.
  Then $\mathfrak{p}^+_n \oplus \mathfrak{p}^-_n=\mathfrak{p}_{n,\mathbb{C}}$.
  
  \begin{dfn}
    Let $L_{n,k}=\mathbb{C}[\mathrm{M}_{n,k}]$ be the space of polynomials
    in the entries of $(n,k)$-matrix $X=(X_{i,j})$ over $\mathbb{C}$.
    \begin{enumerate}
      \item we define the $(\mathfrak{g}_{n,\mathbb{C}},\widetilde{K_n})$-module structure $l_{n,k}$ on $L_{n,k}$ as follows:
      \begin{align*}
        l_{n,k}(\kappa_{i,j})&=\sum_{s=1}^k X_{i,s}\frac{\partial}{\partial X_{j,s}}+\frac{k}{2}\delta_{i,j},\\
        l_{n,k}(\pi^{+}_{i,j})&=\sqrt{-1}\sum_{s=1}^k X_{i,s}X_{j,s},\\
        l_{n,k}(\pi^{-}_{i,j})&=\sqrt{-1}\sum_{s=1}^k \frac{\partial^2}{\partial X_{i,s}\partial X_{j,s}}.
      \end{align*}
      For $(g,\epsilon)\in \widetilde{\mathrm{U}_n(\mathbb{C})} \cong \widetilde{K_n}$
      $(g \in \mathrm{U}_n(\mathbb{C}), \epsilon \in \{\pm 1\})$ and $f(X)\in L_{n,k}$,
      we define
      \[l_{n,k}((g,\epsilon))f(X)=\epsilon^kf(^tgX)\]
      \item we define the left action of the orthogonal group $\mathrm{O}_k$ on $L_{n,k}$ by the right transition.
    \end{enumerate}
    This representation $(l_{n,k}, L_{n,k})$ is well-defined and we call it the Weil representation.
  
  \end{dfn}
  
  For a representation $(\lambda,V_\lambda)$ in the unitary dual $\widehat{\mathrm{O}_k}$ of $\mathrm{O}_k$,
  we put $L_{n,k}(\lambda)=\mathrm{Hom}_{\mathrm{O}_k}(V_\lambda,L_{n,k})$ and
  induce $(\mathfrak{g}_{n,\mathbb{C}},\widetilde{K_n})$-module structure from that of $L_{n,k}$ to it.
  We denote by $L(\tau)$ the unitary lowest weight $(\mathfrak{g}_{n,\mathbb{C}},\widetilde{K_n})$-module
  with lowest $\widetilde{K_n}$-type $\tau$.
  Let $(\tau,U_\tau)$ be the highest weight module of $\widetilde{K_n}$ with highest module $\tau$ 
  and $(\lambda, V_\lambda)$ be the highest weight module of $\mathrm{O}_k$ with highest module $\lambda$.
  \begin{rem}\quad
    \begin{itemize}
    \item Note that sometimes we identify the irreducible representation of $K_n\cong \mathrm{U}(n)$ (resp. $\widetilde{K_n}\cong \widetilde{\mathrm{U}(n)}$)
    with the  finite dimensional irreducible representation of $\mathrm{GL}_n(\mathbb{C})$ (resp. $\widetilde{\mathrm{GL}_n(\mathbb{C})}$). 
    \item $(\lambda, V_\lambda)$ is not always uniquely determined by the highest weight alone.
    The specific expression follows from Howe \cite[\S 3.6.2]{howe1995perspectives}.
  \end{itemize}
    
  \end{rem}
  
  The following symbols are provided to represent the decomposition of $L_{n,k}$.
  \begin{dfn}
    Let $\Delta_{n,k}$ be the set of Young diagrams $D=(\mu_1, \mu_2, \ldots) $
    whose depth $l(D)$ satisfies  $l(D)\leq\min\{n,k\}$ and if $l(D)>k/2$ then $\mu_j=1$ ($k-l(D)< j \leq l(D)$).
    We put $\mathbbm{1}_n=(\underbrace{1,\ldots,1}_n)$
    and $\emptyset=(0,\cdots,0) \in \Delta_{n,k}$.
    For $D \in \Delta_{n,k}$,
    we define 
    \[\tau_{n,k}(D)=(\mu_1+\tfrac{k}{2},\ldots,\mu_n+\tfrac{k}{2}),\qquad
    \lambda_k(D)= \left\{
      \begin{array}{ll}
        (\mu_1,\ldots,\mu_k) &(l(D)\leq k/2)\\
        (\mu_1,\ldots,\mu_{k-l(D)}) &(l(D)>k/2)
      \end{array}
    \right.. \]
  
  \end{dfn}
  
  \begin{thm}\label{thm:howe}
    \begin{enumerate}
      \item We have $L_{n,k}(\lambda)\neq 0$ if and only if $\lambda=\lambda_k(D)$ for some $D \in \Delta_{n,k}$.
      \item The lowest $\widetilde{K_n}$-type of $L_{n,k}(\lambda_k(D))$ is $\tau_{n,k}(D)$.
      \item Under the joint action of $(\mathfrak{g}_{n,\mathbb{C}},\widetilde{K_n}) \times \mathrm{O}_k$, we have
    \[L_{n,k}\cong\bigoplus_{D \in \Delta_{n,k}}L(\tau_{n,k}(D))\boxtimes V_{\lambda_k(D)}.\]
    \end{enumerate}    
  
  \end{thm}
  These results are proved by Kashiwara-Vergne \cite{Kashiwara1978Segal}, and Howe \cite{howe1995perspectives}.
  From this, we get correspondence between the highest weights of $\mathrm{GL}_n(\mathbb{C})$
  and those of $\mathrm{O}_k$, which is called Howe duality.
  \vskip\baselineskip
  
  We fix positive integers $n_1,\ldots,n_d$ and set $n=n_1+\cdots+n_d$.
  We embed $\widetilde{G_{n_1}}\times \cdots \times \widetilde{G_{n_d}}$
  (resp. $\mathfrak{g}_{n_1,\mathbb{C}}\oplus\cdots\oplus\mathfrak{g}_{n_d,\mathbb{C}}$,
  $\widetilde{K_{n_1}}\times\cdots\times \widetilde{K_{n_d}}$)
  diagonally into $\widetilde{G_n}$ (resp. $\mathfrak{g}_{n,\mathbb{C}}$, $\widetilde{K_{n}}$).
  We denote its image by $\widetilde{G_n'}$ (resp. $\mathfrak{g}'_\mathbb{C}$, $\widetilde{K_n'}$).
  We denote by $X^{(s)}$ the indeterminate of $L_{n_s}$.
  Then, we can easily check that the $\mathbb{C}$-isomorphism
  \[\bigotimes_{s=1}^{d}L_{n_s,k}\cong L_{n,k}\]
  given by $X^{(s)}_{i,j}\mapsto X_{n_1+\cdots+n_{s-1}+i,j}$ is the isomorphism as
  $(\mathfrak{g}'_\mathbb{C}$, $\widetilde{K'})\times {\mathrm{O}_k}^d$-module.
  We identify $\bigotimes_{s=1}^d (l_{n_s,k},L_{n_s,k})$ with $(l_{n,k},L_{n,k})$.
  
  \subsection{Pluriharmonic polynomials}
  \begin{dfn}
    If polynomial $f(X) \in L_{n,k}$ satisfies
    \[\qquad l_k(\pi^-_{i,j})f=\sum_{s=1}^k \frac{\partial^2f}{\partial X_{i,s}\partial X_{j,s}}=0 \text{  for any  } i,j\in\{1,\ldots,n\}, \]
    we say that $f(X)$ is pluriharmonic polynomial for $\mathrm{O}_k$.\\
    We denote the set of all pluriharmonic polynomials for $\mathrm{O}_k$ in $L_{n,k}$ by $\mathfrak{H}_{n,k}$.
  
  \end{dfn}
  
  The following proposition is given by Kashiwara-Vergne \cite{Kashiwara1978Segal}.
  
  \begin{prop}
    \begin{enumerate}
      \item $L_{n,k}=L_{n,k}^{\mathrm{O}_k}\cdot \mathfrak{H}_{n,k}$.
      \item $L_{n,k}^{\mathrm{O}_k}$ is the subspace $\mathbb{C}[Z^{(n)}]$ of polynomials in the entries of  $(n, n)$-symmetric matrix $Z^{(n)} =X ^tX$.
      \item Under the joint action of $\widetilde{K_n} \times \mathrm{O}_k$, we have
        \[\mathfrak{H}_{n,k}\cong \bigoplus_{D \in \Delta_{n,k}}U_{\tau_{n,k}(D)}\boxtimes V_{\lambda_k(D)}.\]
    \end{enumerate}
  
  \end{prop}
  
  We denote by $\mathfrak{H}_{n,k}(D)$ the subspace of $\mathfrak{H}_{n,k}$
  corresponding to $U_{\tau_{n,k}(D)}\boxtimes V_{\lambda(D)}$ under this isomorphism.
  
  \begin{lem}\label{lem:pluriharmonic}
    \[\mathrm{Hom}_{(\mathfrak{g'}_\mathbb{C}, \widetilde{K'})}
    \left( \bigotimes_{s=1}^{d}  L(\tau_{n_s,k}(D)), L(\tfrac{k}{2}\mathbbm{1}_n)\right)
    \cong \left(\left(\bigotimes_{s=1}^{d} \mathfrak{H}_{n_s,k}(D)\right)^{\mathrm{O}_k}
    \otimes \left(\bigotimes_{s=1}^{d} U^*_{\tau_{n_s,k}(D)}\right)\right)^{\widetilde{K'}}.
    \]
  \end{lem}
  \begin{proof}
     We note that $L_{n,k}(\emptyset)=L_{n,k}^{\mathrm{O}_k}
     \cong L(\tfrac{k}{2}\mathbbm{1}_n)$ by Theorem~\ref{thm:howe}. We have
     \begin{eqnarray*}\mathrm{Hom}_{(\mathfrak{g}'_\mathbb{C}, \widetilde{K'})}
        \left(\bigotimes_{s=1}^{d}  L(\tau_{n_s,k}(D)), L_{n,k}\right)
        &\cong&  \bigotimes_{s=1}^{d} \mathrm{Hom}_{(\mathfrak{g}_s, \widetilde{K_s})}
        \left(L(\tau_{n_s,k}(D)),L_{n_s,k}\right)\\
        &\cong&  \bigotimes_{s=1}^{d} \mathrm{Hom}_{ \widetilde{K_s}}
        \left( U_{\tau_{n_s,k}(D)}, \mathfrak{H}_{n_s,k}\right)\\
        &=&  \bigotimes_{s=1}^{d} \mathrm{Hom}_{ \widetilde{K_s}}
        \left( U_{\tau_{n_s,k}(D)}, \mathfrak{H}_{n_s,k}(D)\right)\\
        &\cong& \bigotimes_{s=1}^{d} \left(\mathfrak{H}_{n_s,k}(D)
        \otimes U^*_{\tau_{n_s,k}(D)}\right)^{\widetilde{K_s}}\\
        &\cong& \left(\left(\bigotimes_{s=1}^{d} \mathfrak{H}_{n_s,k}(D)\right)
        \otimes \left(\bigotimes_{s=1}^{d} U^*_{\tau_{n_s,k}(D)}\right)\right)^{\widetilde{K'}}.
     \end{eqnarray*}
     Restricting to the $\mathrm{O}_k$-invariant subspace gives the desired isomorphism.
  \end{proof}
  
  There is a natural injection 
  
  \begin{align*}
    \left(\bigotimes_{s=1}^{d} \mathfrak{H}_{n_s,k}(D)\right)^{\mathrm{O}_k}
    \otimes \left(\bigotimes_{s=1}^{d} U^*_{\tau_{n_s,k}(D)}\right)
    &\hookrightarrow \left(\bigotimes_{s=1}^{d} L_{n_s,k}(D)\right)^{\mathrm{O}_k}
    \otimes \left(\bigotimes_{s=1}^{d} U^*_{\tau_{n_s,k}(D)}\right)\\
    &\hookrightarrow L_{n,k}^{\mathrm{O}_k}\otimes \left(\bigotimes_{s=1}^{d} U^*_{\tau_{n_s,k}(D)}\right)\\
    &\cong \mathbb{C}[Z^{(n)}]\otimes \left(\bigotimes_{s=1}^{d} U^*_{\tau_{n_s,k}(D)}\right).
  \end{align*}
  
  We denote the image of $h \in \left(\bigotimes_{s=1}^{d} \mathfrak{H}_{n_s,k}(D)\right)^{\mathrm{O}_k}
  \otimes \left(\bigotimes_{s=1}^{d} U^*_{\tau_{n_s,k}(D)}\right)$ by $\Phi_h(Z^{(n)})$.
  
  \subsection{Holomorphic automorphic forms}
  \begin{dfn}
    Let $(\tau, U_\tau)$ be an irreducible unitary representation of $\widetilde{K_{n}}$
    and $\widetilde{\Gamma_n}$ be a discrete subgroup of $\widetilde{G_n}$.
    Then, a holomorphic automorphic form of type $\tau$ for $\widetilde{\Gamma_n}$ is
    a $U^*_\tau$-valued $C^\infty$-function $\phi$ on $\widetilde{G_n}$ which satisfies the following conditions:
    \begin{enumerate}
      \item $\phi(\gamma gk)=\tau^*(k)^{-1}\phi(g)$ for $k\in \widetilde{K_{n}}$ and $\gamma\in \widetilde{\Gamma_n}$,
      \item $\phi$ is annihilated by the right derivation of $\mathfrak{p}^-$,
      \item $\phi $ is of moderate growth.
    \end{enumerate}
  \end{dfn}
  
  We denote the space of moderate growth $C^\infty$-functions on $\widetilde{G_n}$ which
  are invariant under left translation by $\widetilde{\Gamma_n}$ by $C_\mathrm{mod}^\infty(\widetilde{\Gamma_n}\backslash \widetilde{G_n})$
  and the space consisting of all holomorphic automorphic forms of type $\tau$ for $\widetilde{\Gamma_n}$
  by $\left[C_\mathrm{mod}^\infty(\widetilde{\Gamma_n}\backslash \widetilde{G_n})\otimes U_\tau^*\right]^{\widetilde{K_{n}},\mathfrak{p}^-=0}$.
  Note that $\mathrm{Hom}_{\widetilde{K_n}}\left(U_\tau,C_\mathrm{mod}^\infty(\widetilde{\Gamma_n}\backslash \widetilde{G_n})\right)
  \cong\left[C_\mathrm{mod}^\infty(\widetilde{\Gamma_n}\backslash \widetilde{G_n})\otimes U_\tau^*\right]^{\widetilde{K_{n}}},$
  We can obtain the following well-known isomorphism.
  
  \begin{prop}\label{prop:holohom}
    We have
    \[\mathrm{Hom}_{(\mathfrak{g}_{n,\mathbb{C}}, \widetilde{K_n})}\left(L(\tau),C_\mathrm{mod}^\infty(\widetilde{\Gamma_n}\backslash \widetilde{G_n})\right)
    \cong\left[C_\mathrm{mod}^\infty(\widetilde{\Gamma_n}\backslash \widetilde{G_n})\otimes U_\tau^*\right]^{\widetilde{K_{n}},\mathfrak{p}^-=0}.\]
  
  \end{prop}
  
  Under this isomorphism,
  we denote by $I_F\in\mathrm{Hom}_{(\mathfrak{g}_{n,\mathbb{C}}, \widetilde{K_n})}\left(L(\tau),C_\mathrm{mod}^\infty(\widetilde{\Gamma_n}\backslash \widetilde{G_n})\right)$
  the corresponding homomorphism of $F \in \left[C_\mathrm{mod}^\infty(\widetilde{\Gamma_n}\backslash \widetilde{G_n})\otimes U_\tau^*\right]^{\widetilde{K_{n}},\mathfrak{p}^-=0}$.
  Let $\widetilde{\Gamma'}$ be the image of $\widetilde{\Gamma_{n_1}}\times\cdots\times\widetilde{\Gamma_{n_d}}$ in $\widetilde{G_n}$.
  
  \begin{thm}\label{thm:maintype}
    Let F be a holomorphic automorphic form of type $\tfrac{k}{2}\mathbbm{1}_{n}$ for $\widetilde{\Gamma_n}$
    and $D \in \Delta_{n,k}$ be a Young diagram  such that $l(D)\leq\min\{n_1,\ldots,n_d\}$. 
    We put $\pi_{n}^+=(\pi^+_{i,j})\in M_{n}(\mathfrak{p}^+)$.
    We denote by $\mathrm{Res}$ the pullback of the functions on $\widetilde{G_n}$
    by the diagonal embedding $\widetilde{G_{n_1}}\times\cdots\times\widetilde{G_{n_d}} \hookrightarrow \widetilde{G_n}$.
    Then we have 
      $
      \mathrm{Res}\left(\Phi_{h}(\pi_{n}^+)F\right)
      \in
      \bigotimes_{s=1}^d\left[C_\mathrm{mod}^\infty(\widetilde{\Gamma_{n_s}}\backslash \widetilde{G_{n_s}})\otimes
      U_{\tau_{n_s,k}(D)}^*\right]^{\widetilde{K_{n_s}},\mathfrak{p}^-=0}$
      for any  $h \in \left(\left(\bigotimes_{s=1}^{d} \mathfrak{H}_{n_s,k}(D)\right)^{\mathrm{O}_k}
      \otimes \left(\bigotimes_{s=1}^{d} U^*_{\tau_{n_s,k}(D)}\right)\right)^{\widetilde{K'}}$.
  \end{thm}
  \begin{proof}
    We fix an isomorphism $U^*_{\tfrac{k}{2}\mathbbm{1}_n}\cong\mathbb{C}$.
    Let $I_F:L(\tfrac{k}{2}\mathbbm{1}_n)\rightarrow C_\mathrm{mod}^\infty(\widetilde{\Gamma_n}\backslash \widetilde{G_n})$
    be the corresponding $(\mathfrak{g}_{n,\mathbb{C}},\widetilde{K_n})$-homomorphism of holomorphic automorphic form $F$.
    Multiplying the isomorphism $U^*_{\tfrac{k}{2}\mathbbm{1}_n}\cong\mathbb{C}$
    by a non-zero constant if necessarily, we assume that $I_F(1)=F$.
    We put 
    \[h=\sum_{i}h_{i}\otimes w_{i}^*
    \qquad \left(h_i\in\left(\bigotimes_{s=1}^{d} \mathfrak{H}_{n_s,k}(D)\right)^{\mathrm{O}_k},
    \  w_{i}^* \in \bigotimes_{s=1}^{d}  U^*_{\tau_{n_s,k}(D)}\right).\]
    Then, 
    we denote by $I_h: \bigotimes_{s=1}^d L(\tau_{n_s,k}(D))\rightarrow L(\tfrac{k}{2}\mathbbm{1}_n)$
    the corresponding $(\mathfrak{g}'_\mathbb{C}, \widetilde{K'})$-homomorphism of $h$ in the isomorphism of Lemma~\ref{lem:pluriharmonic}.
    For $w \in \bigotimes_{s=1}^dU_{\tau_{n_s,k}(D)}\subset\bigotimes_{s=1}^d L(\tau_{n_s,k}(D))$,
    we have
    \[
      I_h(w)=\sum_{i}\langle w,w_{i}^*\rangle h_{i}
      \in\left(\bigotimes_{s=1}^{d} \mathfrak{H}_{n_s,k}(D)\right)^{\mathrm{O}_k}
      \subset L_{n,k}^{\mathrm{O}_k}
      \cong L(\tfrac{k}{2}\mathbbm{1}_n).
    \]
    On the other hand, by the definition of Weil representation, we have
    \[h_{i}=\Phi_{h_{i}}(Z^{(n)})=\Phi_{h_{i}}(-\sqrt{-1}\pi_{n}^+)\cdot 1=(-\sqrt{-1})^{m}\Phi_{h_{i}}(\pi_{n}^+)\cdot 1,\]
    where $m$ is a degree of $\Phi_{h_{i}}$.
    (Note that nonzero elements of $\mathfrak{H}_{n_s,k}(D)$ consist of the same degree homogeneous polynomial.)
  
    Therefore, 
    {\jot=10pt
    \begin{eqnarray*}
      I_F\circ I_h(w)
      &=& I_F\left(\textstyle\sum_{i}\langle w,w_{i}^*\rangle h_{i}\right)\\
      &=& (-\sqrt{-1})^{m}I_F\left(\textstyle\sum_{i}\langle w,w_{i}^*\rangle\Phi_{h_{i}}(\pi_{n}^+)\cdot 1\right)\\
      &=& (-\sqrt{-1})^{m}\textstyle\sum_{i}\langle w,w_{i}^*\rangle\Phi_{h_{i}}(\pi_{n}^+)\cdot I_F\left(1\right)\\
      &=& (-\sqrt{-1})^{m}\textstyle\sum_{i}\langle w,w_{i}^*\rangle\Phi_{h_{i}}(\pi_{n}^+)\cdot F\\
      &=& \langle w,(-\sqrt{-1})^{m}\left(\textstyle\sum_{i}w_{i}^*\otimes\Phi_{h_{i}}\right)(\pi_{n}^+)\cdot F\rangle\\
      &=& \langle w,(-\sqrt{-1})^{m}\Phi_{h}(\pi_{n}^+)\cdot F\rangle.
    \end{eqnarray*}}
    Under the isomorphism of Proposition~\ref{prop:holohom},
    the $(\mathfrak{g}'_{n,\mathbb{C}},\widetilde{K'_n})$-homomorphism
    $\mathrm{Res} \circ I_F\circ I_h: L(\tau_{n,k}(D))
    \rightarrow C_\mathrm{mod}^\infty(\widetilde{\Gamma'}\backslash \widetilde{G'_n})$
    corresponds to 
    $\mathrm{Res}\left((-\sqrt{-1})^{m}\Phi_{h}(\pi_{n}^+)\cdot F\right)$,
    which is an element of $\bigotimes_{s=1}^d\left[C_\mathrm{mod}^\infty(\widetilde{\Gamma_{n_s}}\backslash \widetilde{G_{n_s}})\otimes
      U_{\tau_{n_s,k}(D)}^*\right]^{\widetilde{K_{n_s}},\mathfrak{p}^-=0}$.
    From the above, we have that $\mathrm{Res}\left(\Phi_{h}(\pi_{n}^+) F\right)$ is in
    $\bigotimes_{s=1}^d\left[C_\mathrm{mod}^\infty(\widetilde{\Gamma_{n_s}}\backslash \widetilde{G_{n_s}})\otimes
      U_{\tau_{n_s,k}(D)}^*\right]^{\widetilde{K_{n_s}},\mathfrak{p}^-=0}$.
  \end{proof}
  
  \vskip\baselineskip
  
  We define the action of $\widetilde{G_n}$ on $\mathbb{H}_n$ to be through $G_n$
  and we also denoted The complexification of the representation $\tau$ of $\widetilde{K_n}\cong \widetilde{U_n}$ by the same symbol.
  
  For a representation $(\tau, U_\tau)$ of $\widetilde{K_n}$, we define the representation $(\tau',U_{\tau'}\ (=U_\tau))$ by 
  $\tau'(g)=\tau(^tg^{-1})$.
  There is an isomorphism of  $\tau^*\cong\tau'$ as representations.
  For $f \in M_{\tau}(\Gamma_n)$ , we define a $U_{\tau}$-valued $C^\infty$-function $\phi_f$ on $\widetilde{G_n}$ by
  \[\phi_f(g)=(f|_{\tau'} g )(\sqrt{-1})=\tau( ^tj(g,\sqrt{-1}))f(g\cdot \sqrt{-1})\quad\text{for}\quad g\in\widetilde{G_n}.\]
  
  \begin{prop}\label{prop:automisom}
    The above correspondence $f\mapsto \phi_f$ gives the isomorphism
    \[M_\tau(\Gamma_n)\overset{\sim}{\longrightarrow}
    \left[C_\mathrm{mod}^\infty(\widetilde{\Gamma_n}\backslash \widetilde{G_n})\otimes U_{\tau'}\right]^{\widetilde{K_{n}},\mathfrak{p}^-=0}.\]
  
  \end{prop}

  By the above isomorphism and Theorem~\ref{thm:maintype},
  We have the following theorem, which can be proved in the same way as Ban \cite[Theorem 4.2.5]{ban2006rankin}.
  See it for a detailed discussion.
  \begin{thm}\label{thm:mainweight}
    Let F be a holomorphic automorphic form of weight $\tfrac{k}{2}\mathbbm{1}_{n}$ for $\Gamma_n$
    and $D \in \Delta_{n,k}$ be a Young diagram  such that $l(D)\leq\min\{n_1,\ldots,n_d\}$. 
    We put $\partial_Z=\left(\frac{1+\delta_{i,j}}{2}\frac{\partial}{\partial z_{i,j}}\right)$.
    We denote by $\mathrm{Res}$ be the restriction of a function
    on $\mathbb{H}_n$ to $\mathbb{H}_{n_1} \times\cdots\times \mathbb{H}_{n_d}$.
    Then we have
      $\mathrm{Res}\left(\Phi_{h}(\partial_Z)F\right)\in \bigotimes_{s=1}^d M_{\tau_{n_s,k}(D)}(\Gamma_{n_s})$ 
       for any $h \in \left(\left(\bigotimes_{s=1}^{d} \mathfrak{H}_{n_s,k}(D)\right)^{\mathrm{O}_k}
      \otimes \left(\bigotimes_{s=1}^{d} U^*_{\tau_{n_s,k}(D)}\right)\right)^{\widetilde{K'}}$.
  \end{thm}
  Before the proof, we provide some notations and a lemma.
  
  \begin{dfn}
  For a holomorphic function $f$ on $\mathbb{H}_n$ and a representation $(\tau, U_\tau)$ of $\widetilde{U_n}$,
  we define the function $\tilde{f}$ on $\widetilde{G_n}$
  and  the representation $(\widetilde{\tau}, U_\tau)$ of $\widetilde{G_n}$ as follow:
  \[\tilde{f}(g)=f(g\cdot\sqrt{-1}) \qquad g \in \widetilde{G_n},\]
  \[\tilde{\tau}((g_0,\epsilon))=\tau((j(g_0,\sqrt{-1}),\epsilon))\qquad g_0 \in G_n,\ \epsilon\in\{\pm 1\} .\]
  \end{dfn}

  For $Z=X+\sqrt{-1}Y\in\mathbb{H}_n$,
  we put $g_{Z,0}^{(n)}=\begin{pmatrix} \sqrt{Y} & X\sqrt{Y}^{-1} \\ 0& \sqrt{Y}^{-1}\\ \end{pmatrix}\in G_n$.
  We define the section of $\widetilde{G_n}\ni g\mapsto g\cdot\sqrt{-1}\in\mathbb{H}_n$ by
  \[Z \mapsto (g_{Z,0}^{(n)},(\det(\sqrt{Y})^\frac{1}{2}))\coloneqq g^{(n)}_Z.\]
  We have 
  \[j(g_{Z,0}^{(n)},\sqrt{-1})=\sqrt{Y}^{-1}\qquad g^{(n)}_{Z}\cdot\sqrt{-1}=Z.\]
  For $D=(\mu_1,\mu_2,\ldots) \in \Delta_{n,k}$,
  we denote by $\rho_{n,D}$ the representation of $\widetilde{GL_n(\mathbb{C})} $ with a dominant integral weight $D$.
  (Then, $\tau_{n,k}(D)=\det^{\frac{k}{2}}\otimes\rho_{n,D}$.)
  
  \begin{lem}[Ban \cite{ban2006rankin}]\label{lem:deri}
    For a holomorphic function $f$ on $\mathbb{H}_n$
    and a representation $(\tau, U_\tau)$ of $\widetilde{U_n}$,
    we have
    \begin{eqnarray*}
    (\pi_{i,j}^+\tilde{f})(g)&=&4(^t\tilde{\omega}(g)\cdot\widetilde{\partial_Zf}(g)\cdot\tilde{\omega}(g))_{i,j},\\
    (\pi_{i,j}^+\tilde{\tau'})(g)&=&-\sqrt{-1}\tilde{\tau'}(g)\cdot d\tau((e_{i,j}+e_{j,i})\cdot\overline{\tilde{\omega}(g)}^{-1}\cdot\tilde{\omega}(g)),
    \end{eqnarray*}
    where  $\tilde{\omega}(g)=\!^tj(g_0,\sqrt{-1})^{-1}$ for $g=(g_0,\epsilon)\in\widetilde{G_n}$.
    
    In particular, for $Z=X+\sqrt{-1}Y$, we have
    \begin{eqnarray*}
    (\pi_{i,j}^+\tilde{f})(g_Z)&=&4(\sqrt{Y}\cdot\widetilde{\partial_Zf}(g_Z)\cdot \sqrt{Y})_{i,j},\\
    (\pi_{i,j}^+\tilde{\tau'})(g_Z)&=&-\sqrt{-1}\tilde{\tau'}(g_Z)\cdot d\tau((e_{i,j}+e_{j,i})).
    \end{eqnarray*}
  \end{lem}

  \begin{proof}[Proof of Theorem~\ref{thm:mainweight}]
    From Theorem~\ref{thm:maintype}, we have 
    \[\mathrm{Res}\left(\Phi_{h}(\pi_{n}^+)\cdot \phi_F\right) \in
    \bigotimes_{s=1}^d\left[C_\mathrm{mod}^\infty(\widetilde{\Gamma_{n_s}}\backslash \widetilde{G_{n_s}})\otimes
      U_{\tau_{n_s,k}(D)}^*\right]^{\widetilde{K_{n_s}},\mathfrak{p}^-=0}.\]
    By Proposition~\ref{prop:automisom}, there exists $f \in \bigotimes_{s=1}^d M_{\tau_{n_s,k}(D)}(\Gamma_{n_s})$ 
    such that $\phi_f=\mathrm{Res}\left(\Phi_{h}(\pi_{n}^+)\cdot \phi_F\right)$.
    Since
    \[\phi_f(g^{(n_1)}_{Z_1},\ldots,g^{(n_d)}_{Z_d})=(\bigotimes_{s=1}^{d} \det(\sqrt{Y_s})^{\frac{k}{2}}(\rho_{n_s,D}(\sqrt{Y_s})))f(Z_1\ldots,Z_d),\]
    we have
    \[f(Z_1,\ldots,Z_d)=(\bigotimes_{s=1}^{d} \det(\sqrt{Y_s})^{-\frac{k}{2}}\rho_{n_s,D}^{-1}(\sqrt{Y_s}))\mathrm{Res}\left(\Phi_{h}(\pi_{n}^+)\cdot \phi_F\right)(g^{(n_1)}_{Z_1},\ldots,g^{(n_d)}_{Z_d}).\]
    Now we consider $\pi_{i,j}^+$'s action on $\phi_F$. Note that $\phi_F(g)=(\widetilde{\det^\frac{k}{2}}\cdot\widetilde{F})(g)$ using the notations above.

    Under the isomorphism $\mathbb{H}_n\times\widetilde{K_n}\cong\widetilde{G_n}$,
    we regard the function $\phi_F$ as the function in $Z=X+\sqrt{-1}Y\in \mathbb{H}_n$, $\sqrt{Y}$ and $k\in $ $\widetilde{K_n}$.
    We can easily check that $d\tau((e_{i,j}+e_{j,i})\cdot\overline{\tilde{\omega}(g)}^{-1}\cdot\tilde{\omega}(g))$ is 
    invariant under the left transition by $g_Z$ for any $Z\in \mathbb{H}_n$
    and this could be regarded as a function on $\widetilde{K_n}$.
    Since the right derivations by $\pi_{i,j}^+$ commutes with the left transition by $g_Z$,
    the derivated function of  $d\tau((e_{i,j}+e_{j,i})\cdot\overline{\tilde{\omega}(g)}^{-1}\cdot\tilde{\omega}(g))$
    could also be regarded as a function on $\widetilde{K_n}$.
  
    From the Lemma~\ref{lem:deri}, we can easily check that
    the highest degree part of $\Phi_{h}(\pi_{n}^+)\cdot \phi_F$ in $\sqrt{Y}$
    is $4^m\det(\sqrt{Y})^{\frac{k}{2}} \cdot \Phi_h(\sqrt{Y}\partial_Z\sqrt{Y})\cdot\phi_F$,
    where $m$ is a degree of $\Phi_F$.
  
    Since $h \in \left(\left(\bigotimes_{s=1}^{d} \mathfrak{H}_{n_s,k}(D)\right)^{\mathrm{O}_k}
    \otimes \left(\bigotimes_{s=1}^{d} U_{\tau'_{n_s,k}(D)}\right)\right)^{\widetilde{K'}}$,
    we have
    \[
      \mathrm{Res}(4^m\det(\sqrt{Y})^{\frac{k}{2}} \cdot \Phi_h(\sqrt{Y}\partial_Z\sqrt{Y})\cdot\phi_F)
      =4^m(\bigotimes_{s=1}^{d} \det(\sqrt{Y_s})^{\frac{k}{2}}(\rho_{n_s,D}(\sqrt{Y_s})))\mathrm{Res}( \Phi_h(\partial_Z)\cdot\phi_F).
    \]
    Thus we may denote $\mathrm{Res}\left(\Phi_{h}(\pi_{n}^+)\cdot \phi_F\right)$ by
    \begin{align*}
      \mathrm{Res}\left(\Phi_{h}(\pi_{n}^+)\cdot \phi_F\right)(g_1,\ldots,g_n)
      =4^m&(\bigotimes_{s=1}^{d} \det(\sqrt{Y_s})^{\frac{k}{2}}(\rho_{n_s,D}(\sqrt{Y_s})))\mathrm{Res}( \Phi_h(\partial_Z)\cdot\phi_F)\\
      &+(\bigotimes_{s=1}^{d} \det(\sqrt{Y_s})^{\frac{k}{2}})R,
    \end{align*}
    where $R=R(Z_1,\ldots,Z_d,\sqrt{Y_1},\ldots,\sqrt{Y_d},k_1,\ldots,k_d)$ is
    a polynomial with a degree strictly lower than
    that of $\bigotimes_{s=1}^{d} \rho_{n_s,D}(\sqrt{Y_s})$ in $(\sqrt{Y_1},\ldots,\sqrt{Y_d})$.
    Then we have,
    \[f=4^m\mathrm{Res}( \Phi_h(\partial_Z)\cdot\phi_F)+(\bigotimes_{s=1}^{d} \rho_{n_s,D}(\sqrt{Y_s})^{-1})R.\]
    On the other hand, since $f$ is a holomorphic function, $R=0$.
    Therefore, $\mathrm{Res}( \Phi_h(\partial_Z)\cdot\phi_F)=4^{-m}f$
    is an element of $\bigotimes_{s=1}^d M_{\tau_{n_s,k}(D)}(\Gamma_{n_s})$.
  \end{proof}
  
  \begin{prop}[Ibukiyama \cite{ibukiyama1999differential}]\label{cor:diff}
    If $d=2$ and $4k\geq n=n_1+n_2$,
    then there exists a differential operator satisfying condition (A)
    and it is unique up to a constant.
  \end{prop}
  
  A more detailed algorithm for the construction of the differential operators is examined by Ibukiyama \cite{Ibukiyama2020generic}.
  
  \subsection{Pullback formula}
  Let $k$, $\nu$, $n_1$, and $n_2$ be positive integers such that $k, \nu$ are even.
  We take a differential operator $\mathbb{D}_\mathbf{k}=\mathbb{D}_{\mathbf{k},n_1,n_2}$ on $\mathbb{H}_{n}$ satisfying Condition (A)
  for $k$ and $\det^k \rho_{n_1,\mathbf{k}}\otimes \det^k \rho_{n_2,\mathbf{k}}$.
  For a fixed dominant integral weight $\mathbf{k}$ and an integer $r$ such that $l(\mathbf{k}) \leq r$,
  we define $\rho_r=\det^k \rho_{r,\mathbf{k}}$.
  For a Hecke eigenform $f \in S_{\rho_r}(\Gamma_r)$, we define $D(s,f)$ by
  \[D(s,f)=\zeta(s)^{-1}\prod_{i=1}^r\zeta(2s-2i)^{-1}L(s-r,f,\mathrm{St}).\]
  Let $\{f_{r, 1}, \ldots,f_{r, d_r}\}$ 
  be a orthogonal basis of $S_{\rho_r}(\Gamma_r)$ consisting of Hecke eigenforms for a positive integer $r$.
  By Taylor \cite[Lemma 2.1]{Taylor1988congruences},
  we can assume that $f_{r, 1}, \ldots,f_{r, d_r}$  are elements in $S_{\rho_r}(\Gamma_r)(\mathbb{Z}).$
  
  The following theorem is well known as the pullback formula.
  
  \begin{thm}[pullback formula]\label{thm:eisen} Let $k$ be a integer such that $k\geq{n_1+n_2+1}$.
    Then we have
    \[\mathbb{D}_{\mathbf{k},n_1,n_2}E_{n_1+n_2,k}\begin{pmatrix}Z& O\\ O & W\\ \end{pmatrix} =
    \sum_{r=1}^{\min\{n_1, n_2\}}c_{r,\mathbf{k}}\sum_{j=1}^{d_r}\frac{D(k,f_{r, j})}{(f_{r,j},f_{r,j})}
    \left[f_{r,j} \right]^{\rho_{n_1}}(Z)\left[ f_{r,j} \right]^{\rho_{n_2}}(W).\]
    for some constant $c_{r,\mathbf{k}}$.
  \end{thm}
  This constant $c_{r,\mathbf{k}}$ is specifically given by Ibukiyama \cite{Ibukiyama2022differantial}.
  The weight $k$ can be made a smaller value, but the value of the constants $c_r$ at that time has not been specifically calculated.
  For more details on the pullback formula, see for example \cite{bocherer1985uber, garrett1984pullbacks, Ibukiyama2022differantial, kozima2008garrett}.

  \subsection{The case of $l(\mathbf{k})=1$}
   the case $\mathbf{k}=(\nu,0,0,\ldots)$ (i.e. $l(\mathbf{k})=1$) has been investigated by 
  B\"{o}cherer-Satoh-Yamazaki \cite{Bocherer1992pullback}, Ibukiyama \cite{Ibukiyama2022differantial} and others.
  We describe their results for later sections.
  
  \begin{dfn}\label{def:gegen}
    We define the polynomial $P_{d,\nu}(s,m)$ by 
    \[\frac{1}{(1-2st+mt^2)^{(d-2)/2}}=\sum_{\nu=0}^{\infty} P_{d,\nu}(s,m)t^\nu.\]
    These polynomials are called Gegenbauer polynomials.
  
  \end{dfn}
  
  The Gegenbauer polynomial $P_{d,\nu}(s,m)$ can be written concretely as follow:
  \[P_{d,\nu}(s,m)=\sum_{\mu=0}^{\left[\nu/2\right]}(-1)^\mu\frac{(d/2-1)_{\nu-\mu}}{(\nu-2\mu)!\mu!}(2s)^{\nu-2\mu}m^\mu,\]
  where $(x)_\mu=x(x+1)\cdots(x+\mu-1)$ and $\left[\ \cdot \ \right]$ is the Gauss symbol.
  
  For $Z \in \mathbb{H}_n$, we write $Z=\begin{pmatrix}Z_1& Z_{12}\\ {}^tZ_{12} & Z_2\\ \end{pmatrix}$
  ($Z_1\in \mathbb{H}_{n_1}$, $Z_2\in \mathbb{H}_{n_2}$, $Z_{12}\in M_{n_1,n_2}(\mathbb{C})$)
  and $Z_m=(z_{i,j}^{(m)})$, $Z_{12}=(z_{i,j}^{(12)})$.
  We define the $n_m\times n_m$ matrix of partial differential operator $\dfrac{\partial}{\partial Z_m}$ by
  \[\frac{\partial}{\partial Z_m}=\left(\frac{1+\delta_{i,j}}{2}\frac{\partial}{\partial z_{i,j}^{(m)}}\right)\]
  and the $n_1\times n_2$ matrix of partial differential operator $\dfrac{\partial}{\partial Z_{12}}$ by
  \[\frac{\partial}{\partial Z_{12}}=\left(\frac{1}{2}\frac{\partial}{\partial z_{i,j}^{(12)}}\right).\]
  
  \begin{prop}[\cite{Bocherer1992pullback}, \cite{Ibukiyama2022differantial}]
    Let $\mathbf{k}=(\nu,0,0,\ldots)$ be a dominant integral weight and $k$ be a non-negative integer.
    We take $u=(u_1,\ldots,u_{n_1})$ and $v=(v_1,\ldots,v_{n_2})$ to be the variables
    in the representation spaces of $\rho_{n_1,\mathbf{k}}$ and $\rho_{n_2,\mathbf{k}}$
    as the symmetric tensor representations.
    Then, 
    \[\mathbb{D}_{k,\nu}=P_{2k,\nu}\left(u\frac{\partial}{\partial Z_{12}} {}^tv,
    \left(u\frac{\partial}{\partial Z_1} {}^tu\right)\left(v\frac{\partial}{\partial Z_2} {}^tv\right)\right)\]
    satisfies Condition (A) for $\det^k$ and $\det^k\rho_{n_1,\mathbf{k}}\otimes\det^k\rho_{n_2,\mathbf{k}}$.
  
  \end{prop}
  
  We put $c_{k,\nu,r}=c_{r, \mathbf{k}}$.
  By using exact pullback formula given by Ibukiyama \cite{Ibukiyama2022differantial},
  we have
  \begin{align*}
    c_{k,\nu, 1}&=2^4(k)_\nu\ (2k-1)_{\nu-1}\cdot\pi,\\
    c_{k,\nu,2}&=(-1)^{\frac{k}{2}}2^8\frac{(k)_\nu\ (2k-1)_{\nu-2}}{k-2}\cdot \pi^3.
  \end{align*}
  
  \section{Kurokawa-Mizumoto congruence}
  
  Let $k$, $\nu$, $n_1$, and $n_2$ be positive integers such that $k, \nu$ are even and $1\leq n_1\leq n_2\leq2$.
  For a dominant integral weight $\mathbf{k}=(\nu,0,0,\ldots)$,
  we define $\rho_n=\det^k \rho_{n,\mathbf{k}}$ and  $c_{k,\nu, r}$ as in the previous chapters.
  Let $\{f_{r, 1}, \ldots,f_{r, d_r}\}$ 
  be a orthogonal basis of $S_{\rho_r}(\Gamma_r)$ consisting of Hecke eigenforms in $S_{\rho_r}(\Gamma_r)(\mathbb{Z})$ for a positive integer $r$.
  
  Now we define the functions as follows:
  \begin{eqnarray*}
    Z(n,k)&=&\zeta(1-k)\prod_{j=1}^{\left[n/2\right]}\zeta(1+2j-2k),\\
    \widetilde{E_{n,k}}(Z)&=&Z(n,k)E_{n,k}(Z),\\
    \mathcal{E}(Z_1,Z_2)&=&\mathcal{E}_{k,\nu,n_1,n_2}(Z_1,Z_2)
    \coloneqq  \nu!(2\pi\sqrt{-1})^{-\nu}\mathbb{D}_{\nu,n_1,n_2}\widetilde{E_{n_1+n_2,k}}\begin{pmatrix}Z_1& O\\ O & Z_2\\ \end{pmatrix}.
  \end{eqnarray*}
  For $f\in S_{k}(\Gamma_1)$, we put 
  \begin{eqnarray*}
    \mathbb{L}(s,f,\mathrm{St})&=&\Gamma_{\mathbb{C}}(s)\Gamma_\mathbb{C}(s+k-1)\frac{L(s,f,\mathrm{St})}{(F,F)},\\
    \mathcal{C}_{m,k}(f)&=&\frac{Z(m,k)}{Z(2,k)}\mathbb{L}(k-1,f,\mathrm{St}).
  \end{eqnarray*}
  For $F\in S_{(k+\nu,k)}(\Gamma_2)$, we put
  \begin{eqnarray*}
    \mathbb{L}(s,F,\mathrm{St})&=&\Gamma_{\mathbb{C}}(s)\Gamma_\mathbb{C}(s+k+\nu-1)\Gamma_\mathbb{C}(s+k-2)\frac{L(s,F,\mathrm{St})}{(F,F)},\\
    \mathcal{C}_{m,k}(F)&=&\frac{Z(m,k)}{Z(4,k)}\mathbb{L}(k-2,F,\mathrm{St}).
  \end{eqnarray*}
  
  \begin{prop}[Kozima \cite{kozima2000on}]\label{prop:sym-algebricity}
    Let $k,\nu\in2\mathbb{Z}_{\geq0}$ and $k\geq2n+2$.
    Let $f\in S_{\det^k\otimes\mathrm{Sym}^\nu}(\Gamma_n)(\mathbb{Q}(f))$ be an Hecke eigenform.
    Let $m\in\mathbb{Z}$ be such that
    $1\leq m \leq k-n \text{ and } m\equiv n \mod 2$. 
    We assume that $n\equiv 3 \mod 4$ if $m=1$.
    Then we have
    \[\frac{L(m,f,\mathrm{St})}{\pi^{nk+l+m(n+1)-n(n+1)/2}(f,f)}\in \mathbb{Q}(f).\]
  \end{prop}
  
    By Proposition~\ref{prop:scalar-algebricity} and Proposition~\ref{prop:sym-algebricity},
    we obtain the following corollary.
  
    \begin{cor}
      \begin{enumerate}
        \item For a Hecke eigenform $f\in S_k(\Gamma_1)(\mathbb{Q}(f))$, we have
          $\mathcal{C}_{4,k}(f)\in \mathbb{Q}(f)$.
        \item For a Hecke eigenform $F\in S_{(k+\nu,k)}(\Gamma_2)(\mathbb{Q}(f))$ with $k\geq 6$, we have
        $\mathcal{C}_{4,k}(F)\in \mathbb{Q}(F)$.
      \end{enumerate}
      
    \end{cor}
  
  From the above theorem, the following proposition follows.
  
  \begin{prop}Let $1\leq n_1 \leq n_2 \leq 2$ and $k\geq 4$. Then
    \[
      \mathcal{E}_{k,\nu,n_1,n_2}(Z_1,Z_2)
      =\sum_{r=1}^{\min\{n_1, n_2\}}\gamma_r\sum_{j=1}^{d_r}
      \mathcal{C}_{n_1+n_2,k}(f_{r,j})\left[f_{r,j}\right]^{\rho_{n_1}}(Z_1)
      \left[ f_{r,j} \right]^{\rho_{n_2}}(Z_2),\]
  where $\gamma_r$ is a $p$-unit rational number for $r=1, 2$ and any prime $p$ with $p\geq 2(k+\nu)-3$.
  \end{prop}
  \begin{proof}
    We only need to determine the $\gamma_r$ for $r=1,2$.
    we have
    \begin{eqnarray*}
      c_{k,\nu, 1}Z(2,k)\frac{D(k,f_{1,j})}{(f_{1,j},f_{1,j})}
      &=& \frac{\zeta(1-k)\zeta(3-2k)}{\zeta(k)\zeta(2k-2)}c_{k,\nu, 1}\cdot \frac{L(k-1,f_{1,j},\mathrm{St})}{(f_{1,j},f_{1,j})}\\
      &=& (-1)^{k/2+1}\frac{\Gamma_\mathbb{C}(k)\Gamma_\mathbb{C}(2k-2)}{\Gamma_\mathbb{C}(k-1)\Gamma_\mathbb{C}(2k+\nu-2)}c_{k,\nu, 1}
      \cdot \mathbb{L}(k-1,f_{1,j},\mathrm{St})\\
      &=& \frac{(k-1)(-1)^{k/2+1}(2\pi)^{\nu-1}}{(2k+\nu-3)_\nu}c_{k,\nu, 1}\cdot \mathbb{L}(k-1,f_{1,j},\mathrm{St}).
    \end{eqnarray*}
    Therefore, we may take $\gamma_1$ as 
  
    \begin{eqnarray*}
    \gamma_1 &=&\nu! \cdot (2\pi\sqrt{-1})^{-\nu} \cdot \frac{(k-1)(-1)^{k/2+1}(2\pi)^{\nu-1}}{(2k+\nu-3)_\nu}\cdot c_{k,\nu, 1}\\
    &=&(-1)^{(k+\nu)/2+1}2^3\nu!\frac{(k-1)_{\nu+1}(2k-1)_{\nu-2}}{(2k+\nu-2)_{\nu-1}}.
    \end{eqnarray*}
  
    By the same calculation, we have
  \[\gamma_2=(-1)^{\nu/2+1}2^5\nu!\frac{(k-1)_{\nu+1}(2k-1)_{\nu-2}}{(2k+\nu-4)_{\nu-1}}.  \]
    Thus, the proposition follows from Theorem~\ref{thm:eisen}.
  \end{proof}
  
  We write $\mathcal{E}(Z_1,Z_2)$ as
  \[\mathcal{E}(Z_1,Z_2)=\sum_{N\in H_{n_2}(\mathbb{Z})_{\geq 0}}g_{(k,\nu,n_1,n_2),N}^{(n_1)}(Z_1)\mathbb{e}(\mathrm{tr}(NZ_2)).\]
  Then, $g_N^{n_1}=g_{(k,\nu,n_1,n_2),N}^{(n_1)}\in M_{\rho_{n_1}}\left(\Gamma_{n_1}\right)\otimes V_{n_2,\mathbf{k}}.$
  From the proposition, we get the following corollary.
  
  \begin{cor}\label{cor:g_N}
    For $N \in H_{n_2}(\mathbb{Z})_{> 0}$ and $1\leq n_1 \leq n_2 \leq 2$ , we have
   \[g_N^{n_1}(Z_1)
   =\sum_{r=1}^{\min\{n_1, n_2\}}\gamma_r\sum_{j=1}^{d_r}
   \mathcal{C}_{n_1+n_2,k}(f_{r,j})\left[f_{r,j}\right]^{\rho_{n_1}}(Z_1)
   a\left(N,\left[ f_{r,j} \right]^{\rho_{n_2}}\right)
   .\]
  \end{cor}
  
  About the rationality of these functions, the following propositions hold. These can be proved in the same way as \cite{atobe2023harder}.
  \begin{cor}\label{cor:g_N rationality} We have
    \[g_{(k,\nu,n_1,n_2),N}^{(n_1)}(Z_1) \in (M_{\rho_{n_1}}(\Gamma_{n_1})\otimes V_{n_2,\mathbf{k}})(\mathbb{Q}).\]
    Moreover, if $p >\max\left\{2k, k+\nu-2\right\}$, then
    \[g_{(k,\nu,n_1,n_2),N}^{(n_1)}(Z_1) \in (M_{\rho_{n_1}}(\Gamma_{n_1})\otimes V_{n_2,\mathbf{k}})(\mathbb{Z}_{(p)}).\]
  
  \end{cor}
  
  \begin{prop}
    If $f\in S_{k+\nu}(\Gamma_1)(\mathbb{Q})$, then $\left[f\right]^{(k+\nu,k)}\in S_{(k+\nu,k)}(\Gamma_2)(\mathbb{Q})$.
  
  \end{prop}
    
  Now we define the integral ideal $\mathfrak{A}(f)$ of $\mathbb{Q}(f)$
  for a Hecke eigenform $f\in S_{(k_1,\ldots,k_n)}(\Gamma_n)(\mathbb{Q}(f))$ with $k_n>n+1$
  and state the integrality lemma in \cite{mizumoto1996on, mizumoto1996corrections}.
    
  We put $V=\bigoplus_{\tau}\mathbb{C}f^\tau$, where $\tau$ runs over all embeddings of $\mathbb{Q}(f)$ into $\mathbb{C}$.
  Let $V^\perp$ be the orthogonal complement of $V$ in $S_k(\Gamma_n)$.
  Let $\nu(f)$ (resp. $\kappa(f)$) be the exponent of the finite abelian group $S_k(\Gamma_n)(\mathbb{Z})/(V(\mathbb{Z})\oplus V^\perp(\mathbb{Z}))$
  (resp. $\mathcal{O}_{\mathbb{Q}(f)}/\mathbb{Z}[\lambda_f(T)|T\in \mathcal{H}_n]$).
  we put 
  \[\mathfrak{A}(f)=\kappa(f)\nu(f)\mathfrak{d}(\mathbb{Q}(f)),\]
  where $\mathfrak{d}(\mathbb{Q}(f))$ be the different of $\mathbb{Q}(f)/\mathbb{Q}$.
  
  \begin{lem}[Integrality lemma \cite{mizumoto1996on}]\label{lem:integrality}
    Let $f\in S_k(\Gamma_n)(\mathbb{Q}(f))$ be a Hecke eigenform with
    $n\in \mathbb{Z}_{>0}$ and even integer $k$ such that $k\geq\frac{3}{2}(n+1)$.
    Suppose that $f$ has a Fourier coefficient which is equal to 1.
    Let $K$ be an algebraic number field.
    Then for any $g \in S_k(\Gamma_n)(\mathcal{O}_K)$ we have
    \[\frac{(f,g)}{(f,f)}\in \mathfrak{A}(f)^{-1}\cdot\mathcal{O}_{K\cdot\mathbb{Q}(f)}\]
  \end{lem}

  Using the integrality lemma (Lemma \ref{lem:integrality}) for a normalized Hecke eigenform $f \in S_{k+\nu}(\Gamma_1)$ and 
  $\mathop{\mathrm{pr}}(g_{(k,\nu,1,2),N}^{(1)}(Z_1))$,
  where $\mathrm{pr}: M_{k+\nu}(\Gamma_n) \rightarrow S_{k+\nu}(\Gamma_n)$ be the orthogonal projection,
  we have the following proposition.
  
  \begin{prop} \label{prop:klingen-integrality}
    Let $k,\nu$ be positive even integers with $k\geq 6$, $f \in S_{k+\nu}(\Gamma_1)$ be a normalized Hecke eigenform.
    For any prime $p$ with $p >\max\left\{2k, k+\nu-2\right\}$ and $A \in H_{2}(\mathbb{Z})_{>0}$, we have
    \[a(A,[f]^{\rho_2})\in V_{2,\mathbf{k}}\left(\frac{Z(4,k)}{\gamma_1\mathcal{C}_{4,k}(f)}\mathfrak{A}(f)^{-1}\cdot\mathbb{Z}_{(p)}\right)\]
  \end{prop}
  
  \begin{thm}\label{thm:main}
    Let $k,\nu$ be positive even integers with $k\geq 6$,
    $f_{1,1}=f, \ldots, f_{1,d_1}$ be a basis of $S_{k+\nu}\left(\Gamma_1\right)$ consist of normalized Hecke eigenforms,
    $p$ be a prime number of $\mathbb{Q}$
    and $A \in H_{2}(\mathbb{Z})_{>0}$ be a half-integral positive definite matrix of degree $2$.
    Suppose that $A$ and $p$ satisfy the following conditions:
    \begin{enumerate}
      \item $\mathrm{ord}_p(\mathbb{L}(k-1,f,\mathrm{St}))=:\alpha>0$,
      \item $\mathrm{ord}_p(\mathcal{C}_{4,k}(f)a(A,[f]^{(k+\nu,k)}))= 0$,
      \item $p\geq2(k+\nu)-3$.
    \end{enumerate}
    Then, there is a Hecke eigenform $G \in M_{\rho_2}(\Gamma_2)$
     such that G is not a scalar multiple of $\left[f\right]^{(k+\nu,k)} and $
    \[\left[f\right]^{(k+\nu,k)}\equiv_{ev} G \mod \mathfrak{p}\]
    for some prime ideal $\mathfrak{p}\mid p $ of $\mathbb{Q}(G)$.
    If $\mathrm{ord}_p(\gamma_1) = 0$, Condition (3) can be changed to Condition (3)':
    \begin{enumerate}
      \renewcommand{\labelenumi}{(\arabic{enumi})'}
      \setcounter{enumi}{2}
      \item $p\geq \mathrm{max}\left\{2k,k+\nu-2\right\}$.
    \end{enumerate}
  
    If moreover $k\geq 6$ and $p$ satisfy the following conditions:
    \begin{enumerate}
      \renewcommand{\labelenumi}{(\arabic{enumi})}
      \setcounter{enumi}{3}
      \item $\mathrm{ord}_p(\mathbb{L}(k-1,f_{1,i},\mathrm{St}))\leq 0 \ (2\leq i \leq d_1)$,
      \item $p$ is coprime with every $\mathfrak{A}(f_{r,i})$ ($1\leq r \leq2$, $1\leq i \leq d_r$).
    \end{enumerate} 
    there is a Hecke eigenform $G \in S_{\rho_2}(\Gamma_2)$
     such that G is not a scalar multiple of $\left[f\right]^{(k+\nu,k)} and $
    \[\left[f\right]^{(k+\nu,k)}\equiv_{ev} G \mod \mathfrak{p}^\alpha\]
    for some prime ideal $\mathfrak{p}\mid p $ of $\mathbb{Q}(G)$.
  \end{thm}
  
  \begin{rem}
    Before the proof, We make a few comments on the theorem.
    \begin{itemize}
      \item condition (3) in the main theorem could be loosened a bit more.
      In fact, when $\left(k,\nu\right)=\left(6,12\right)$ and $p=13$,
      numerical calculations estimate that there will be a congruence.
      \item Whether conditions (1) and (2) are valid when there is a congruence is a delicate question.
      It has been suggested that this question is connected to the Bloch-Kato conjecture and
      is not easily proven.
      \item It is known by Katsurada-Mizumoto \cite{katsurada2012congruences}
      that there is an example in the case of scalar values
      where the congruence disappears when conditions (2) do not hold,
      even if conditions (1) and (3) hold.
      It is unknown that if there are similar examples for vector valued cases.
      \item The second half of the theorem shows a congruence modulo power of prime,
      but we have yet to find a numerical example where the latter part of the theorem is effective.
      (Theorem~\ref{example} is an example of the congruence of prime modulo power of prime,
      but this can be shown only with the first half part of the main theorem.)
    \end{itemize}
  \end{rem}
  
  \begin{proof}
    From the assumptions,
    \begin{align*}
      &\mathrm{ord}_p\left(\zeta(3-2k)\ \gamma_1 \ \mathcal{C}_{4,k}(f)a(A,\left[f\right]^{(k+\nu,k)})a(A,\left[f\right]^{(k+\nu,k)})\right)\\
      = \ &\mathrm{ord}_p\left(\gamma_1\cdot
      \frac{\mathcal{C}_{4,k}(f)a(A,\left[f\right]^{(k+\nu,k)})
      \cdot\mathcal{C}_{4,k}(f)a(A,\left[f\right]^{(k+\nu,k)})}
      {\mathbb{L}(k-1,f,St)}\right)
      = -\alpha < 0.
    \end{align*}
      On the other hand, by Von Staudt–Clausen theorem,
      $\mathrm{ord}_{p}\left(\zeta(3-2k)\right)=\mathrm{ord}_p\left(\frac{B_{2k-2}}{2k-2}\right)\geq 0$.
      Thus, the Lemma~\ref{lem:cong}, Corollary~\ref{cor:g_N} and Corollary~\ref{cor:g_N rationality}
      give the first part of the theorem.
  
      \vskip\baselineskip
  
      From Corollary~\ref{cor:g_N} and Corollary~\ref{cor:g_N rationality}, we have
      \[
        \sum_{r=1}^{2}\gamma_r\sum_{j=1}^{d_r}
        \mathcal{C}_{4,k}(f_{r,j})\left[f_{r,j}\right]^{\rho_2}(Z_1)
        a(A,\left[ f_{r,j} \right]^{\rho_2})
        \equiv 0 \mod \mathbb{Z}_{(p)}.
      \]
      Here the congruence is understood to be the system of congruences for Fourier coefficients.
      Under the conditions (4)-(5),
      by using Proposition~\ref{prop:klingen-integrality} , we can calculate as above to obtain 
      \[
        \gamma_1\mathcal{C}_{4,k}(f_{1,j})\left[f_{1,j}\right]^{\rho_2}(Z_1)a(A,\left[ f_{1,j} \right]^{\rho_2})
        \equiv 0 \mod \mathbb{Z}_{(p)}.
      \]
      for any integer $i$ with $2\leq j \leq d_1$.
      Thus we have
      \begin{equation}\label{1}
        \gamma_1
        \mathcal{C}_{4,k}(f)\left[f\right]^{\rho_2}(Z_1)
        a\left(A,\left[ f \right]^{\rho_2}\right)
        +\gamma_2\sum_{j=1}^{d_2}
        \mathcal{C}_{4,k}(f_{2,j})f_{2,j}(Z_1)
        a(A,f_{2,j})
        \equiv 0 \mod \mathbb{Z}_{(p)}.
      \end{equation}
  
      Let $\{v_1,\ldots,v_{\nu+1}\}$ be a fixed basis of $V_{2,(k+\nu,k)}$ and put
      \[a(A,f_{2,j})=a_{j,1}v_1+\cdots+a_{j,\nu+1}v_{\nu+1},\]
      where $a_{j,i} \in \mathbb{Q}$.
      Multiplying each $f_{2,j}$ by an element of $\mathbb{Q}^\times$ and renumber the subscripts if necessary,
      we can assume that 
      \[
        \left\{
        \begin{array}{ll}
          a_{j,1}&=1 \quad (1\leq j \leq d'),\\
          a_{j,1}&=0 \quad (d'+1 \leq j \leq d_2).
        \end{array}
        \right.
      \]
      and
      \begin{equation}\label{2}
        \mathrm{ord}_\mathfrak{P}(\gamma_2
        \mathcal{C}_{4,k}(f_{2,1})a(A,f_{2,1})^2)
        =\mathrm{ord}_\mathfrak{P}(\gamma_2
        \mathcal{C}_{4,k}(f_{2,1}))
        \leq -\alpha
      \end{equation}
      since
      \[
        \mathrm{ord}_p\left(
          \gamma_2\sum_{j=1}^{d_2}
        \mathcal{C}_{4,k}(f_{2,j})
        a(A,f_{2,j})^2
        \right)
        = \mathrm{ord}_p\left( \gamma_1 \ \mathcal{C}_{4,k}(f)a(A,\left[f\right]^{\rho_2})^2\right)
        =-\alpha.
      \]
      We note that $\mathcal{C}_{4,k}(f_{2,j})a(A,f_{2,j})^2$ remains unchanged
      if $f_{2,j}$ is replaced by $\gamma f_{2,j}$ with any $\gamma\in \mathbb{C}$.
  
      For any $T\in \mathcal{H}_n$ we act $T-\lambda_{\left[f\right]^{\rho_2}}(T)$ on the both sides of (\ref{1}),
      we have
      \[
        H(Z_1)\coloneqq \gamma_2\sum_{j=1}^{d_2}\left(\lambda_{f_{2,j}}(T)-\lambda_{\left[f\right]^{\rho_2}}(T)\right)
        \mathcal{C}_{4,k}(f_{2,j})f_{2,j}(Z_1)
        a(A,f_{2,j})
        \equiv 0 \mod \mathbb{Z}_{(p)},
      \]
      since $T$ preserve the $p$-integrality of the Fourier coefficients.
      By \cite{Taylor1988congruences}, we can take $p$-unit $u$ such that
      \[uH\in S_{(k+\nu,k)}(\Gamma_2)(\mathbb{Z}).\]
  
      \begin{lem}\label{lem:vector-integrality}
        Let $F\in S_{(k+\nu,k)}(\Gamma_2)(\mathbb{Q}(F))$ be a Hecke eigenform.
        Suppose that $F$ has a Fourier coefficient $a(N,F)=a_1v_1+\cdots+a_{\nu+1}v_{\nu+1}$
        with $a_i =1$ for some $i$.
        Then for any $G \in S_{(k+\nu,k)}(\Gamma_2)(\mathbb{Z})$, we have
        \[\frac{(F,G)}{(F,F)}\in \mathfrak{A}(F)^{-1}.\]
      \end{lem}
      \begin{proof}
        This proof can be done in the same way as for the scalar valued case \cite{mizumoto1996on}.
      \end{proof}
      Applying Lemma~\ref{lem:vector-integrality} with $F=f_{2,1}$, $G=uH$, we have
      \[
        u\gamma_2\left(\lambda_{f_{2,1}}(T)-\lambda_{\left[f\right]^{\rho_2}}(T)\right)
        \mathcal{C}_{4,k}(f_{2,1})
        a(A,f_{2,1})\in \mathfrak{A}(f_{2,1})^{-1}.
      \]
      Therefore (\ref{2}) gives
      \[\lambda_{f_{2,1}}(T)\equiv \lambda_{\left[f\right]^{\rho_2}}(T) \mod \mathfrak{p}^\alpha\]
      for any prime $\mathfrak{p}\mid p$ in $\mathbb{Q}(f_{2,1})$
      since $\mathfrak{A}(f_{2,1})$ is coprime by the condition (5).
      This completes the proof of the main theorem.
  \end{proof}
  
  In the rest of this chapter, we will examine the conditions that the $G$ in Theorem~\ref{thm:main}
  is a cusp form by using Chenevier-Lannes's method \cite{chenevier2019automorphic}.
  
  Let $\mathcal{O}$ be the ring of integers in an algebraic number field $K$,
  and we take $\mathfrak{p}$ be a maximal ideal of $\mathcal{O}$.
  Let $A_\mathfrak{p}$ be a Grothendieck ring of finite-dimensional continuos representations of
  $\mathrm{Gal}(\bar{\mathbb{Q}}/\mathbb{Q})$ with coefficients
  in $\mathcal{O}/\mathfrak{p}$ unramified outside $\mathfrak{p}$.
  Let $\mathcal{S}$ be the set of isomorphism classes of the simple representations of $A_\mathfrak{p}$.
  For $H=\sum_{S\in\mathcal{S}}n_SS$ ($n_S\in\mathbb{Z}$), we set 
  \[\|H\| =\sum_{S\in\mathcal{S}}|n_S|\dim S.\]
  Let $\chi_\mathfrak{p}: \mathrm{Gal}(\bar{\mathbb{Q}}/\mathbb{Q})\rightarrow\mathrm{GL}_1(K_\mathfrak{p})$ be the cyclotomic character
  and $\overline{\chi_\mathfrak{p}}$ be the mod $\mathfrak{p}$ representation of $\chi_\mathfrak{p}$. 
  The subscript $\mathfrak{p}$ may be omitted if there is no confusion.
  
  \begin{lem}\label{lem:chi}
  Let $j$ be an integer.
  If an element $H \in A_\mathfrak{p}$ satisfies $(1+\overline{\chi}^i)H=0$,
  then $i\|H\|$ is divisible by $(p_\mathfrak{p}-1)$.
  \end{lem}
  
  \begin{proof}
    This lemma for $i =1$ has been proved by Chenevier and Lannes \cite{chenevier2019automorphic}.
    The general case can be proved in the same way, but for readers' convenience we give a proof.
  
    Let $C_{\overline{\chi}}$ be a cyclic subgroup of $A_\mathfrak{p}^\times$ generated by $\overline{\chi}$.
    For $S \in \mathcal{S}$, we denote by $\Omega(S)$ the orbit of $S$ under the action of $C_{\overline{\chi}}$,
    and let $m_i(S)$ be the least integer $k\geq 1$ such that we have $\overline{\chi}^{ik}S=S$.
  
    We fix an element $S \in \mathcal{S}$. we put $d=m_1(S)/m_i(S)\in \mathbb{Z}$.
    We consider $1+\overline{\chi}^i$ as an endomorphism in $\mathrm{End}(\mathbb{Z}[\Omega(S)])$.
    It is easy to see that when $m_i(S)$ is odd, we have $\mathrm{Ker}(1+\overline{\chi}^i)=0$,
    and when $m_i(S)$ is even, $\mathrm{Ker}(1+\overline{\chi}^i)$ is generated by
    $\{(1-\overline{\chi}^i+\cdots-\overline{\chi}^{(m_i(S)-1)i})\overline{\chi}^jS\}_{j=0}^{j=d-1}$.
    Let $\mathcal{S}_{i}$ be the subset of $\mathcal{S}$ consisting of an element $S\in \mathcal{S}$ such that $m_i(S)$ is even.
  
    From the above discussion, if $(1+\overline{\chi}^i)H=0$, then 
    H can be denoted as 
    \[H=\sum_{S\in\mathcal{S}_{i}}\sum_{j=1}^{d}n_{S,j}(1-\overline{\chi}^i+\cdots-\overline{\chi}^{(m_i(S)-1)i})\overline{\chi}^jS,\]
    where $n_{S,j}\in \mathbb{Z}$.
    By the definition, we have
    \[\|H\|=\sum_{j=1}^{d}(\sum_{j=1}^{d}|n_{S,j}|)m_i(S)\dim S.\]
    On the other hand, $im_i(S)\dim S$ is divisible by $(p_\mathfrak{p}-1)$.
    In fact, we have
    \[\det S=\det((\overline{\chi}^i)^{m_i(S)}S)=\overline{\chi}^{im_i(S)\dim S}\det S \]
    and the order of $\overline{\chi}\in A_\mathfrak{p}^\times$ is $p_\mathfrak{p}-1$.
    Therefore, we have $i\|H\|$ is divisible by $(p_\mathfrak{p}-1)$.
  \end{proof}
  
  \begin{thm}\label{thm:cusp}
    We consider under the conditions $(1)\sim(3)$ of Theorem~\ref{thm:main}.
    If $f\in S_{k+\nu}(\Gamma_1)$ is not Hecke congruent with $f_{1,i} \ (i=2,\ldots,d_1)$ respectively,
    and if $4(k-2)$ is not divided by $(p-1)$,
    then $G$ in Theorem~\ref{thm:main} is a cusp form in $S_{(k+\nu,k)}(\Gamma_2)$.
  \end{thm}
  \begin{proof}
    From Proposition~\ref{prop:arakawa1}, it is enough to show that  
    if $\left[f_1\right]^{(k+\nu,k)}\equiv_{ev} \left[f_2\right]^{(k+\nu,k)} \mod  p$ for Hecke eigenforms $f_1,f_2\in S_{k+\nu}(\Gamma_1)$,
    then $f_1 \equiv_{ev} f_2 \mod p$.
  
    Let $\rho_{f_{1,i}} : \mathrm{Gal}(\bar{{\mathbb{Q}}}/\mathbb{Q})\rightarrow \mathrm{GL}_2(\mathbb{Q}_p)$ be 
    the Galois representation attached to the spin L functions of $f_{1,i}$ and
    $\overline{\rho_{f_{1,i}}}$ be the mod $p$ representation of $\rho_{f_{1,i}}$.
    If $\left[f_1\right]^{(k+\nu,k)}\equiv_{ev} \left[f_2\right]^{(k+\nu,k)} \mod  p$,
    we have
    \[(1+\overline{\chi}^{-(k-2)})\overline{\rho_{f_1}}=(1+\overline{\chi}^{-(k-2)})\overline{\rho_{f_2}}\]
    in $A_p$ by Proposition~\ref{prop:arakawa2}.
    Hence, applying Lemma~\ref{lem:chi} to $i=k-2$ and $H=\overline{\rho_{f_1}}-\overline{\rho_{f_2}}$,
    we see that $\overline{\rho_{f_1}}=\overline{\rho_{f_2}}$, since $\|H\|=4 \text{ or } 0$.
    This implies that $\lambda_{f_1}(p)\equiv \lambda_{f_2}(p) \mod q'$ for any prime $q\neq p $.
    This congruence is obviously true for $q=p$.
  \end{proof}
  
  \section{Applications}
  In this section, we consider the conditions that appear in Theorem~\ref{thm:main}.
  We give a way to compute the special value $\mathbb{L}(k-1,f,\mathrm{St})$ of
  L-function appearing in condition (1) by the Petersson inner product
   and give the necessary conditions for condition (2).
  
  \subsection{Condition (1)}
  For $r, N\in \mathbb{Z}_{\geq 0}$, we define $H(r,N)$ by
  \[H(r,N)=\left\{
    \begin{array}{ll}
    \zeta(1-2r) & (N=0) \\
    L(1-r,\left(\frac{-N}{\cdot}\right)) & (N>0, \ N \equiv 0,3 \mod 4)
    \end{array}
  \right.
  \]
  as in \cite{Cohen1975sums}.
  Let $P_{k,r}(t,m)$ be the Gegenbauer polynomial defined in Definition~\ref{def:gegen}.
  
  \begin{thm}[Cohen \cite{Cohen1975sums}, Theorem 6.2]
    Let $r,k$ be positive integers such that $3\leq r\leq k-1$, $r$ is odd and $k$ is even, and set
    \[C_{k,r}(z)=\sum_{m=0}^{\infty}\left(\sum_{\substack{t\in \mathbb{Z}\\t^2\leq 4m}}P_{2r+2,k-r-1}(t/2,m)H(r,4m-t^2)\right)q^m
    \quad (z \in \mathbb{H}_1).\]
    Then, $C_{k,r}\in M_k(\Gamma_1)$.
    Moreover, if $r <k-1$, $C_{k,r} \in S_k(\Gamma_1)$.
  
  \end{thm}
  
  Petersson inner product of $C_{k,r}(z)$ and other Hecke eigenform holds information on the special value of the L-function,
  which has been investigated by Zagier \cite{zagier1977modular}.
  
  \begin{thm}[Zagier \cite{zagier1977modular}, Theorem 2]
    Let $r,k$ be positive integers such that $3\leq r\leq k-1$, $r$ is odd and $k$ is even.
    For any a Hecke eigenform $f\in S_k(\Gamma_1)$,
    \[(f, C_{k,r})=-\frac{(r+k-2)!(k-2)!}{(k-r-1)!}\cdot\frac{1}{4^{r+k-2}\pi^{2r+k-1}}L(r,f,\mathrm{St}).\]
  
  \end{thm}
  
  Thus, for a Hecke eigenform $f\in S_{k+\nu}\left(\Gamma_1\right)$,
  \[(f,C_{k+\nu,k-1})=-\frac{(2k+\nu-3)!(k+\nu-2)!}{\nu!}\cdot\dfrac{1}{4^{2k+\nu-3}\pi^{3k+\nu-3}}\mathrm{L(k-1,f,\mathrm{St})},\]
  and we find
  \begin{align*}
    \mathbb{L}(k-1,f,\mathrm{St}) &=\Gamma_\mathbb{C}(k-1)\Gamma_{\mathbb{C}}(2k+\nu-2)\frac{L(k-1,f,\mathrm{St})}{(f,f)}\\
    &=-\frac{\nu!(k-2)!}{(k+\nu-2)!}\cdot 2^{k+\nu-3}\cdot\dfrac{(f,C_{k+\nu,k-1})}{(f,f)}.
  \end{align*}
  
  \subsection{Condition (2)}
  We define a Hecke operator $T^{(m)}$ for $m=p_1\cdots p_r$ (prime decomposition) by 
  \[T^{(m)}=T(p_1)\cdots T(p_r).\]
  Note that if $p_1,\ldots,p_r$ are different from each other, $T^{(m)}=T(m)$.
  We write 
  \[g_N(z)=g_{(k,\nu,1,2)}^{(N)}(z)=\sum_{n\in\mathbb{Z}_{>0}}\epsilon_{k,\nu}(n,N)q^n,\]
  and
  \[g_N|T^{(m)}(z)=\sum_{n\in\mathbb{Z}_{>0}}\epsilon_{k,\nu}(m,n,N)q^n.\]
  
  Let $\left\{f_j\right\}_{j=1}^d$ be a basis of $S_{k+\nu}\left(\Gamma_1\right)$ consist of normalized Hecke eigenforms,
  and we set $f_j|T^{(m)}=\lambda_{j,m}f_j$. 
  By Corollary~\ref{cor:g_N}, we get the following proposition.
  \begin{prop}
    We have
    \[\epsilon_{k,\nu}(m,n,N)
    =\gamma\sum_{j=1}^{d}\lambda_{j,m}\mathcal{C}_{3,k}(f_j)a(n,f_j)\overline{a(N,\left[f_j\right]^{(k+\nu,k)})}.\]
  
  \end{prop}
  
  Note that $\mathcal{C}_{4,k}(f)=\zeta(3-2k)\mathcal{C}_{3,k}(f)$, the following propositions follow by a simple calculation.
  
  \begin{prop}\label{prop:cond2}
    For $N \in \mathbb{H}_2(\mathbb{Z})_{>0}$, we define $e_m=\epsilon_{k,\nu}(m,1,N)$.
    Let ${f_1=f,\ldots,f_d}$ be a basis of $S_{k+\nu}\left(\Gamma_1\right)$ consist of Hecke eigenforms.
    For $m_1,\ldots, m_d \in \mathbb{Z}_{>0}$, we set $\Delta=\Delta(m_1,\ldots,m_d)=\det(\lambda_{j,m_j})$.
    Then,
    \[\Delta\gamma \mathcal{C}_{4,k}(f)\overline{a(N,[f]^{\rho_2})}=\zeta(3-2k)
    \begin{vmatrix} 
      e_{m_1} & \lambda_{2,m_1} & \dots  & \lambda_{d,m_1} \\
      \vdots & \vdots & \ddots & \vdots \\ 
      e_{m_d} & \lambda_{2,m_d} & \dots  & \lambda_{d,m_d} \\
    \end{vmatrix}.\]
  
  \end{prop}
  
  \begin{cor}
    Assume that a prime ideal $\mathfrak{p}$ of $\mathbb{Q}$ satisfies
    $\mathfrak{p}>\mathrm{max}\left\{2k,k+\nu-2\right\}$ and $\gamma$ is $\mathfrak{p}$-integer.
    Suppose that $\mathfrak{p}$ divides neither $\zeta(3-2k)$ nor 
    $\begin{vmatrix} 
      e_{m_1} & \lambda_{2,m_1} & \dots  & \lambda_{d,m_1} \\
      \vdots & \vdots & \ddots & \vdots \\ 
      e_{m_d} & \lambda_{2,m_d} & \dots  & \lambda_{d,m_d} \\
    \end{vmatrix}$.
    Then
    \[\mathrm{ord}_\mathfrak{p}\left(\mathcal{C}_{4,k}(f)\overline{a(N,[f]^{\rho_2})}\right)\leq 0.\]
  \end{cor}
  \begin{proof}
    It follows from Proposition~\ref{prop:cond2} since $\Delta$ and $\gamma$ are $\mathfrak{p}$-integers.
  \end{proof}
  
  Fourier coefficients $a(T,\widetilde{E_{n,k}})$ of the Siegel-Eisenstein Series $\widetilde{E_{n,k}}(Z)$
  are considered by Katsurada \cite{katsurada1999explicit}. The remaining part of this chapter will summarize the results of the study. 
  
  We define $\chi_p(a)$ and a polynomial $\gamma_p(B,X)$
  for $a \in \mathbb{Q}_p^\times$ and a nondegenerate matrix $B \in H_n(\mathbb{Z}_p)$ by
  \begin{eqnarray*}
  \chi_p(a)&=&\left\{
    \begin{array}{lll}
    1 & (\mathbb{Q}_p(\sqrt{a})=\mathbb{Q}_p)\\
    -1 & (\mathbb{Q}_p(\sqrt{a})/\mathbb{Q}_p \text{ is quadratic unramified})\\
    0 & (\mathbb{Q}_p(\sqrt{a})/\mathbb{Q}_p \text{ is quadratic ramified})
    \end{array}
  \right.,\\
  \gamma_p(B,X)&=&\left\{
    \begin{array}{ll}
      (1-X)\prod_{i=1}^{n/2}(1-p^{2i}X^2)(1-p^{n/2}\chi_p\left((-1)^{n/2}\det B\right)X)^{-1} &(\text{$n$ is even})\\
      (1-X)\prod_{i=1}^{(n-1)/2}(1-p^{2i}X^2) &(\text{$n$ is odd})
    \end{array}
  \right..
  \end{eqnarray*}
  For $B \in H_n(\mathbb{Z})$ with $n$ even, let $\mathfrak{d}_B$ be the discriminant of $\mathbb{Q}(\sqrt{(-1)^{n/2}\det B})/\mathbb{Q}$
  and $\chi_B=(\frac{\mathfrak{d}_B}{\cdot})$ be the Kronecker character.
  
  Let $b_p(B,s)$ be the local Siegel series for an element $B \in H_n(\mathbb{Z}_p)$.
  We define the polynomial $F_p(B,X) \in \mathbb{Z}[X]$ with constant term 1 by 
  \[b_p(B,s)=\gamma_p(B,p^{-s})F_p(B,p^{-s}).\]
  
  For $T \in H_n(\mathbb{Z}_p)\backslash \left\{0\right\}$ (resp. $T \in H_n(\mathbb{Z})_{\geq 0}\backslash \left\{0\right\}$),
  there exists a nondegenerate matrix $\tilde{T} \in H_m(\mathbb{Z}_p)$ (resp. $\tilde{T} \in H_n(\mathbb{Z})_{> 0}$) such that 
  $T$ is similar to  $\begin{pmatrix} 
    \tilde{T} & O \\
    O&O
  \end{pmatrix}$ over $\mathbb{Z}_p$ (resp. $\mathbb{Z}$).
  Using this $\tilde{T}$,
  define $F_p^*(T,X) \in \mathbb{Z}[X] $ for a matrix $T \in H_n(\mathbb{Z}_p)$ by $F_p^*(T,X)=F_p(\tilde{T},X)$,
  and $\chi_T^*$ for $T \in H_n(\mathbb{Z})_{>0}$ with even rank by $\chi_T^*=\chi_{\tilde{T}}$.
  
  \begin{prop}
    Let $k \in 2\mathbb{Z}$. Assume that $k \geq (n+1)/2$ and that neither $k=(n+2)/2 \equiv 2 \mod 4$ nor $k=(n+3)/2 \equiv 2 \mod 4$.
    Then for $T \in H_n(\mathbb{Z})_{\geq 0}$ of rank $m$, we have 
    \begin{eqnarray*}
    a(T,\widetilde{E_{n,k}})&=&2^{[(m+1)/2]}\prod_{p|\det(2\tilde{T})}F_p^*(T,p^{k-m-1})\\
    &\times&  \left\{
    \begin{array}{ll}
      \prod_{i=m/2+1}^{[n/2]}\zeta(1+2i-2k)L(1+m/2-k,\chi_T^*) &(\text{$m$ is even})\\
      \prod_{i=(m+1)/2}^{[n/2]}\zeta(1+2i-2k) &(\text{$m$ is odd})
    \end{array}
    \right..
    \end{eqnarray*}
    Here we make the convention $F_p^*(T,p^{k-m-1})=1$ and $L(1+m/2-k,\chi_T^*)=\zeta(1-k)$ if $m=0$.
  
  \end{prop}
  
  We take a variables $x,y$ of $\rho_{(k+\nu,k)}\coloneqq \det^k\otimes\mathrm{Sym}^\nu$.
  (Then, $\{x^\nu,x^{\nu-1}y,\ldots,y^\nu\}$ is a basis of $\det^k\otimes\mathrm{Sym}^\nu$.)
  We put $v={}^t(x,y)$, and $r(n,N,R)=\mathop{\mathrm{rank}}\begin{pmatrix}
    n &R/2\\ 
    {}^t\!R/2 & N
  \end{pmatrix}\in\left\{2,3\right\}$.
  We can easily find the following proposition.
  
  \begin{prop}
    \[
      \epsilon_{k,\nu}(n,N)(v)
      =\sum_{R\in M_{1,2}(\mathbb{Z})}a\left(T_{(n,N,R)},\widetilde{E_{3,k}}\right)\cdot P_{2k,\nu}(R/2\cdot v,n{}^t\!vNv),
    \]
  where $T_{(n,N,R)}=\begin{pmatrix}
      n &R/2\\ 
      {}^t\!R/2 & N
    \end{pmatrix}$ and $P_{2k,\nu}$ is the Gegenbauer polynomial.
  \end{prop}
  
  \section{Examples}
  \subsection{$(k,\nu)=(14,2)$}
  It is known that $\mathrm{dim}\ S_{16}(\Gamma_1)=\mathrm{dim}\ S_{(16,14)}(\Gamma_2)=1$, and we take
  normalized Hecke eigenforms $\Delta_{16} \in S_{16}(\Gamma_1)$
  and $\Delta_{16,14} \in S_{(16,2)}(\Gamma_2)$.
  We can calculate that
  \[\mathbb{L}(13,\Delta_{16},\mathrm{St})=\frac{2^{20}\cdot3^4\cdot373}{7}.\]
  Therefore 373 is the only prime that satisfies the condition(3) in Theorem\ref{thm:main}.
  Let $n=1,N=\begin{pmatrix}
    1&0\\
    0&1
  \end{pmatrix}$, we get
  \begin{eqnarray*}
  \epsilon_{14,2}\left(1,\begin{pmatrix}
    1&0\\
    0&1
  \end{pmatrix}\right)\left(\begin{pmatrix}
    1\\
    0
  \end{pmatrix}\right)&=&-5291173154072\neq 0 \mod 373,\\
  \gamma&=&-\frac{91}{2147483648} \neq 0 \mod 373.
  \end{eqnarray*}
  By theorem\ref{thm:main} we prove the following theorem.
  \begin{thm}
    \[\Delta_{16,14}\equiv_{ev}[\Delta_{16}]^{(16,14)} \mod 373. \]
    In particular,
    \[\lambda_{\Delta_{16,14}}(p) \equiv (1+p^{12})\lambda_{\Delta_{16}}(p) \mod 373. \]
  
  \end{thm}
  
  \subsection{$(k,\nu)=(8,8)$}
  As previous subsection, $\mathrm{dim}\ S_{16}(\Gamma_1)=\mathrm{dim}S_{(16,8)}(\Gamma_2)=1$, and we take
  normalized Hecke eigenforms $\Delta_{16,8} \in S_{(16,8)}(\Gamma_2)$.
  We can calculate that
  \[\mathbb{L}(7,\Delta_{16},\mathrm{St})=\frac{2^{15}\cdot23^2}{11\cdot13}.\]
  Therefore 23 is the only prime that satisfies the condition(3) in Theorem\ref{thm:main}.
  Let $n=1,N=\begin{pmatrix}
    1&0\\
    0&1
  \end{pmatrix}$, we get
  \begin{eqnarray*}
  \epsilon_{8,8}\left(1,\begin{pmatrix}
    1&0\\
    0&1
  \end{pmatrix}\right)\left(\begin{pmatrix}
    1\\
    1
  \end{pmatrix}\right)&=&-46666368\neq 0 \mod 23,\\
  \gamma&=&-\frac{945945}{2143483648} \neq 0 \mod 23.
  \end{eqnarray*}
  Noting that $\mathrm{dim}\ S_{16}(\Gamma_1)=1$,
  by a simple improvement of Lemma~\ref{lem:cong} and the first half of theorem~\ref{thm:main},
  we can prove the following theorem without using the second half of the theorem~\ref{thm:main}.
  \begin{thm}\label{example}
    \[\Delta_{16,8}\equiv_{ev}[\Delta_{16}]^{(16,8)} \mod 23^2. \]
    In particular,
    \[\lambda_{\Delta_{16,8}}(p) \equiv (1+p^{6})\lambda_{\Delta_{16}}(p) \mod 23^2. \]
  
  \end{thm}
  
  \section{Acknowledgment}
  The author would like to thank T. Ikeda for his great guidance and support as my supervisor,
  and T. Ibukiyama and H. Katsurada for their many suggestions and comments.

  \bibliography{kuromizu}

\begin{thebibliography}{10}

\bibitem{and}
A~N Andrianov.
\newblock Euler products corresponding to siegel modular forms of genus 2.
\newblock {\em Russian Mathematical Surveys}, 29(3):45, jun 1974.

\bibitem{arakawa1983vector}
Tsuneo Arakawa.
\newblock Vector valued siegel's modular forms of degree two and the associated andrianov l-functions.
\newblock {\em Manuscripta mathematica}, 44(1):155--185, 1983.

\bibitem{atobe2023harder}
Hiraku Atobe, Masataka Chida, Tomoyoshi Ibukiyama, Hidenori Katsurada, and Takuya Yamauchi.
\newblock Harder's conjecture i.
\newblock {\em Journal of the Mathematical Society of Japan}, 75(4):1339--1408, 2023.

\bibitem{ban2006rankin}
Katsuma Ban.
\newblock On rankin-cohen-ibukiyama operators for automorphic forms of several variables.
\newblock {\em Commentarii mathematici Universitatis Sancti Pauli}, 55(2):149--171, 2006.

\bibitem{bergstrom2014siegel}
Jonas Bergstr\"{o}m, Carel Faber, and Gerard van~der Geer.
\newblock Siegel modular forms of degree three and the cohomology of local systems.
\newblock {\em Selecta Math. (N.S.)}, 20(1):83--124, 2014.

\bibitem{bocherer1985uber}
Siegfried B\"{o}cherer.
\newblock \"{U}ber die {F}ourier-{J}acobi-{E}ntwicklung {S}iegelscher {E}isensteinreihen. {II}.
\newblock {\em Math. Z.}, 189(1):81--110, 1985.

\bibitem{Bocherer1992pullback}
Siegfried B\"{o}cherer, Takakazu Satoh, and Tadashi Yamazaki.
\newblock On the pullback of a differential operator and its application to vector valued {E}isenstein series.
\newblock {\em Comment. Math. Univ. St. Paul.}, 41(1):1--22, 1992.

\bibitem{chenevier2019automorphic}
Ga{\"e}tan Chenevier and Jean Lannes.
\newblock Automorphic forms and even unimodular lattices.
\newblock {\em Ergebnisse der Mathematik und ihrer Grenzgebiete}, 3, 2019.

\bibitem{Cohen1975sums}
Henri Cohen.
\newblock Sums involving the values at negative integers of {$L$}-functions of quadratic characters.
\newblock {\em Math. Ann.}, 217(3):271--285, 1975.

\bibitem{garrett1984pullbacks}
Paul~B. Garrett.
\newblock Pullbacks of {E}isenstein series; applications.
\newblock In {\em Automorphic forms of several variables ({K}atata, 1983)}, volume~46 of {\em Progr. Math.}, pages 114--137. Birkh\"{a}user Boston, Boston, MA, 1984.

\bibitem{howe1995perspectives}
Roger Howe.
\newblock Perspectives on invariant theory: {S}chur duality, multiplicity-free actions and beyond.
\newblock In {\em The {S}chur lectures (1992) ({T}el {A}viv)}, volume~8 of {\em Israel Math. Conf. Proc.}, pages 1--182. Bar-Ilan Univ., Ramat Gan, 1995.

\bibitem{ibukiyama1999differential}
Tomoyoshi Ibukiyama.
\newblock On differential operators on automorphic forms and invariant pluri-harmonic polynomials.
\newblock {\em Comment. Math. Univ. St. Paul.}, 48(1):103--118, 1999.

\bibitem{Ibukiyama2020generic}
Tomoyoshi Ibukiyama.
\newblock Generic differential operators on {S}iegel modular forms and special polynomials.
\newblock {\em Selecta Math. (N.S.)}, 26(5):Paper No. 66, 50, 2020.

\bibitem{Ibukiyama2022differantial}
Tomoyoshi Ibukiyama.
\newblock Differential operators, exact pullback formulas of {E}isenstein series, and {L}aplace transforms.
\newblock {\em Forum Math.}, 34(3):685--710, 2022.

\bibitem{Kashiwara1978Segal}
M.~Kashiwara and M.~Vergne.
\newblock On the {S}egal-{S}hale-{W}eil representations and harmonic polynomials.
\newblock {\em Invent. Math.}, 44(1):1--47, 1978.

\bibitem{katsurada2012congruences}
H.~Katsurada and S.~Mizumoto.
\newblock Congruences for {H}ecke eigenvalues of {S}iegel modular forms.
\newblock {\em Abh. Math. Semin. Univ. Hambg.}, 82(2):129--152, 2012.

\bibitem{katsurada1999explicit}
Hidenori Katsurada.
\newblock An explicit formula for {S}iegel series.
\newblock {\em Amer. J. Math.}, 121(2):415--452, 1999.

\bibitem{katsurada2008congruence}
Hidenori Katsurada.
\newblock Congruence of {S}iegel modular forms and special values of their standard zeta functions.
\newblock {\em Math. Z.}, 259(1):97--111, 2008.

\bibitem{kozima2000on}
Noritomo Kozima.
\newblock On special values of standard {$L$}-functions attached to vector valued {S}iegel modular forms.
\newblock {\em Kodai Math. J.}, 23(2):255--265, 2000.

\bibitem{kozima2008garrett}
Noritomo Kozima.
\newblock Garrett's pullback formula for vector valued {S}iegel modular forms.
\newblock {\em J. Number Theory}, 128(2):235--250, 2008.

\bibitem{kurokawa1979congruences}
Nobushige Kurokawa.
\newblock Congruences between {S}iegel modular forms of degree two.
\newblock {\em Proc. Japan Acad. Ser. A Math. Sci.}, 55(10):417--422, 1979.

\bibitem{mizumoto1996corrections}
Shin-ichiro Mizumoto.
\newblock Corrections to: ``{O}n integrality of {E}isenstein liftings''.
\newblock {\em Manuscripta Math.}, 90(2):267--269, 1996.

\bibitem{mizumoto1986congruences}
Shinichiro Mizumoto.
\newblock Congruences for eigenvalues of {H}ecke operators on {S}iegel modular forms of degree two.
\newblock {\em Math. Ann.}, 275(1):149--161, 1986.

\bibitem{mizumoto1991poles}
Shinichiro Mizumoto.
\newblock Poles and residues of standard {$L$}-functions attached to {S}iegel modular forms.
\newblock {\em Math. Ann.}, 289(4):589--612, 1991.

\bibitem{mizumoto1996on}
Shinichiro Mizumoto.
\newblock On integrality of {E}isenstein liftings.
\newblock {\em Manuscripta Math.}, 89(2):203--235, 1996.

\bibitem{satoh1986certain}
Takakazu Satoh.
\newblock On certain vector valued {S}iegel modular forms of degree two.
\newblock {\em Math. Ann.}, 274(2):335--352, 1986.

\bibitem{shimura2000arithmeticity}
Goro Shimura.
\newblock {\em Arithmeticity in the theory of automorphic forms}, volume~82 of {\em Mathematical Surveys and Monographs}.
\newblock American Mathematical Society, Providence, RI, 2000.

\bibitem{Taylor1988congruences}
Richard~Lawrence Taylor.
\newblock {\em On congruences between modular forms}.
\newblock ProQuest LLC, Ann Arbor, MI, 1988.
\newblock Thesis (Ph.D.)--Princeton University.

\bibitem{zagier1977modular}
D.~Zagier.
\newblock Modular forms whose {F}ourier coefficients involve zeta-functions of quadratic fields.
\newblock In {\em Modular functions of one variable, {VI} ({P}roc. {S}econd {I}nternat. {C}onf., {U}niv. {B}onn, {B}onn, 1976)}, volume Vol. 627 of {\em Lecture Notes in Math.}, pages 105--169. Springer, Berlin-New York, 1977.

\end{thebibliography}
  \bibliographystyle{plain}
  
  \end{document}